\documentclass[oneside]{amsart}
\usepackage{amsfonts}
\usepackage{amssymb}
\usepackage{amsxtra}
\usepackage{amstext}
\usepackage[english]{babel}
\usepackage{url}

\newcommand{\R}{\mathbb{R}}

\newcommand{\D}{\mathbb{D}}
\newcommand{\N}{\mathbb{N}}
\newcommand{\E}{\mathbb{E}}
\newcommand{\HH}{\mathbb{H}}
\newcommand{\V}{\mathbb{V}}
\newcommand{\Z}{\mathbb{Z}}

\newcommand{\pp}{\mathbb{P}}
\newcommand{\kA}{\mathcal{A}}
\newcommand{\kD}{\mathcal{D}}

\newcommand{\kO}{\mathcal{O}}

\newcommand{\kF}{\mathcal{F}}

\newcommand{\kL}{\mathcal{L}}

\def\l{\ell}
\def\eps{\varepsilon}

\newtheorem {theo} {Theorem} [section]

\newtheorem{lem}[theo]{Lemma}
\newtheorem{prop}[theo]{Proposition}
\newtheorem{cor}[theo]{Corollary}
\newtheorem{rem}[theo]{Remark}

\ifx\hyperlink\undefined
\newcommand{\hurl}[2]{\url{#1}}
\else
\newcommand{\hurl}[2]{\href{#2}{#1}}
\fi

\title[Excited Brownian motions as limits of excited random walks]
      {Excited Brownian motions as limits of excited random walks}

\author{Olivier RAIMOND}
\address{Laboratoire Modal'X, Universit\'e Paris Ouest Nanterre La D\'efense, B\^atiment G, 200 avenue de la R\'epublique 92000 Nanterre, France.}
\email{olivier.raimond@u-paris10.fr}
\author{Bruno SCHAPIRA}
\address{D\'epartement de Math\'ematiques, B\^at. 425, Universit\'e Paris-Sud 11, F-91405 Orsay, cedex, France. }
\email{bruno.schapira@math.u-psud.fr}

\begin{document}

\begin{abstract} 
We obtain the convergence in law of a sequence of excited (also called cookies) random walks toward an excited Brownian motion. This last process is a continuous semi-martingale whose drift is a function, say $\varphi$, of its local time. 
It was introduced by Norris, Rogers and Williams as a simplified version of Brownian polymers, and then recently further studied by the authors. To get our results we need 
to renormalize together the sequence of cookies, the time and the space in a convenient way. 
The proof follows a general approach already taken by T\'oth and his coauthors in multiple occasions, which goes through Ray-Knight type results. 
Namely we first prove, when $\varphi$ is bounded and Lipschitz, that the convergence holds at the level of the local time processes. This is done via a careful study of the transition kernel of an auxiliary Markov chain which describes the local time at a given level. Then we prove a tightness result and deduce the convergence at the level of the full processes. 
\end{abstract}

\maketitle

\section{Introduction} 
\subsection{General overview}
Self-interacting random processes play a prominent role in the probability theory and in statistical physic.  
One fascinating aspect is that behind an apparent simplicity, they can be extremely  
hard to analyze rigorously.  
Just to mention one striking example, it is still not known whether 
once reinforced random walks on a ladder are recurrent in general (see however \cite{Sel} and \cite{Ver} for a partial answer and the surveys  \cite{MR} and \cite{Pem} for other problems on reinforced processes). A major difficulty in these models is that we loose the Markovian property and in particular the usual dichotomy between recurrence and transience can be broken. 
A famous example where this happens is for vertex reinforced random walks on $\Z$: it is now a well known 
result in the field, first conjectured and partially proved by Pemantle and Volkov \cite{PemV}, that almost surely these processes eventually get stuck on five sites \cite{Tar}. For analogous results concerning self-attracting diffusions, see \cite{CLJ}, \cite{HRo} and \cite{R}.

Beside this very basic, yet fundamental, problem of recurrence, a question of particular interest is to understand the connections between the various 
discrete and continuous models. In particular an important challenging conjecture is that self-avoiding random walks on $\Z^2$ converge, after renormalization, toward the $SLE_{8/3}$ (see \cite{LSW} for a discussion on this and \cite{DCS} for some recent progress). 
There are in fact not many examples where invariance principles or central limit theorems were fully established. 
But for instance it was proved that random walks perturbed at extrema converge after the usual renormalization toward a perturbed Brownian motion (see e.g. \cite{Dav} and \cite{Wer}).

In this paper we are interested in the class of so-called excited random processes, 
which are among the most elementary examples of self-interacting processes. By this we mean that 
the interaction with the past trajectory is as localized as possible: 
the evolution of these processes at any time only depend on their local time at their present position.  
A discrete version was introduced relatively recently by 
Benjamini and Wilson \cite{BW} and a generalization, called multi-excited or cookie random walks, was then further studied 
by Zerner \cite{Z} and many other authors (see in particular \cite{MPRV} and references therein). Closely related models were also considered in \cite{ABK}, \cite{BKS}, \cite{K} and \cite{KRS}. 
Dolgopyat \cite{D} observed that in dimension $1$, in the recurrent regime, and after the usual renormalization, multi-excited random walks also converge 
toward a perturbed Brownian motion (we will give a more precise statement later). 
However, as we will see below, the latter are not, in some sense, the most natural continuous versions of excited processes.  
Somewhat more natural ones were introduced two decades ago by Norris, Rogers and Williams \cite{NRW2}, in connection with the excluded volume problem \cite{NRW1}, and as a simplified model  
for Brownian polymers. They were later called excited Brownian motions by the authors \cite{RS}.

\vspace{0.2cm} 
The aim of this paper is to show that excited Brownian motions can be approached in law by multi-excited random walks in the Skorokhod space, i.e. in the sense of the full process. 
For this we need to use a nonstandard renormalization, namely we need to scale together and appropriately the sequence of cookies, which govern the drift of the walk, the space and the time. Now let us give more details, starting with some definitions:

\vspace{0.2cm}
A multi-excited or cookie random walk $(X_\eps(n),n\ge 0)$ is associated to a sequence   
$$\varepsilon:=(\varepsilon_i,i\ge 1) \in (-1,1)^\N,$$
of cookies in the following way: set 
$$p_{\eps,i}:=\frac{1}{2}(1+\varepsilon_i),$$
for all $i\ge 1$, and let $(\kF_{\eps,n},n\ge 0)$ be the filtration generated by $X_\eps$. Then $X_\eps(0):=0$ and for all $n\ge 0$,
$$\pp[X_\eps(n+1)-X_\eps(n) = 1 \mid \kF_{\eps,n}]= 1 -  \pp[X_\eps(n+1)-X_\eps(n) = -1 \mid \kF_{\eps,n}] = p_{\eps,i},$$
if $\#\{j\le n\ :\ X_\eps(j)=X_\eps(n) \} = i$. We notice that the case of random cookies has also been studied in the past, for instance by Zerner \cite{Z}, but here we consider only deterministic $\varepsilon$. 

\vspace{0.2cm}
\noindent On the other hand excited Brownian motions are solutions of a stochastic differential equation of the type: 
$$dY_t=dB_t + \varphi(L_t^{Y_t})\, dt,$$
where $B$ is a Brownian motion, $L_\cdot^\cdot $ is the local time process of $Y$ and $\varphi:\R\to \R$ is some measurable (bounded) function. 

\vspace{0.2cm}
So at a heuristic level the discrete and the continuous models are very similar: the drift is a function of the local time at the present position. 
But the analogy can be pushed beyond this simple observation. In particular criteria for recurrence and nonzero speed in both models (see respectively \cite{KZ} and \cite{RS}) are entirely similar (see below). 
Our results here give now a concrete link. 
We first prove that when $\varphi$ is bounded and Lipschitz, the local time process of $X_\eps$, conveniently renormalized, converges to the one of $Y$, 
exactly in the same spirit as in T\'oth's papers on self-interacting random walks (see \cite{T}). Then we obtain a tightness result and deduce a convergence in the Skorokhod space at the level of the processes (see Theorem \ref{cor3} below).  
For proving the convergence of the local time processes we use a standard criterion of Ethier and Kurtz \cite{EK} on approximation of diffusions. To show that we can apply it here  
we introduce some auxiliary Markov chains describing the local time on each level  
and we make a careful analysis of the transition kernels of these Markov chains (see Section \ref{sectheo}). 

\vspace{0.2cm} 
\noindent \textit{Acknowledgments: We thank an anonymous referee for encouraging us to prove the main result without the recurrence hypothesis and with weaker assumptions on the regularity of $\varphi$. } 

\subsection{Description of the results}
\noindent For $a\in \Z$ and $v\in \N$, set 
$$\tau_{\eps,a}(v):= \inf\left\{ j \ :\ \#\{i\le j\ :\ X_\eps(i)=a\text{ and } X_\eps(i+1)=a-1\}=v+1\right\}.$$
Consider the process $(S_{\eps,a,v}(k),k\in \Z)$ defined by 
$$S_{\eps,a,v}(k)=\#\{j\le \tau_{\eps,a}(v)-1\ :\ X_\eps(j)=k \text{ and } X_\eps(j+1)=k-1\}.$$
In particular, when $\tau_{\eps,a}(v)$ is finite, $S_{\eps,a,v}(a)=v$. We will say that $X_\eps$ is recurrent when all the $\tau_{\eps,a}(v)$'s are finite. 
A criterion for recurrence when $\eps_i\ge 0$ for all $i$ or when $\eps_i = 0$ for $i$ large enough is given in \cite{Z} and \cite{KZ} (namely in these cases, $X_\eps$ is a.s. recurrent if, and only if, $\sum_i \eps_i \in [-1,1]$).

Assume now that $\varphi$ is bounded and let $\varepsilon_n=(\varepsilon_{i}(n),i\ge 1)$ be defined by
$$\varepsilon_{i}(n) := \frac 1 {2n} \varphi\left(\frac{i}{2n}\right)\quad \text{for all }n\ge 1 \text{ and all }i\ge 1.$$
Since $\varphi$ is bounded, if $n$ is large enough then $\eps_n \in (-1,1)^\N$ and $X_{\eps_n}$ is well defined. 
Then for $a\in \R$ and $v\ge 0$, set
\begin{eqnarray}
\label{LambdaS}
\Lambda^{(n)}_{a,v}(x) :=\frac{1}{n} S_{\eps_n,[2na],[nv]}([2nx])\quad \text{for all }x\in \R.
\end{eqnarray}
We give now the analogous definitions in the continuous setting. First for $a\in \R$, let 
$$\tau_a(v) := \inf\{t>0 \ :\ L_t^a > v\}\quad \text{for all }v\ge 0,$$
be the right continuous inverse of the local time of $Y$ at level $a$. 
Again we say that $Y$ is a.s. recurrent if all these stopping times are a.s. finite. This is equivalent (see \cite[Theorem 1.1]{RS}) to the condition $C_1^+=C_1^-=+\infty$, where
$$C_1^\pm:=\int_0^\infty \exp \left[ \mp\int_0^x h(z) \frac{dz}{z} \right] \ dx,$$ 
and where
$$h(z):=\int_0^z\varphi(\l)\, d\l \quad \text{for all }z\ge 0.$$
In particular when $\varphi$ is nonnegative or compactly supported this is equivalent to $\int_0^\infty \varphi(\l)\ d\l\in[-1,1]$. 
Then set 
$$\Lambda_{a,v}(x):=L_{\tau_{a}(v)}^{x}.$$ 
The Ray-Knight theorem describes the law of $(\Lambda_{a,v}(x),x\in \R)$ (a proof is given in \cite{NRW2} when $v=0$, but it applies as well for $v>0$), when $\tau_a(v)$ is a.s. finite, and we recall this result now. To fix ideas we assume that $a\le 0$. An analogous result holds for $a\ge 0$. 
So first we have $\Lambda_{a,v}(a)=v$. Next, $\Lambda_{a,v}$ is solution of the stochastic differential equation: 
\begin{eqnarray}
\label{EDS}
d\Lambda_{a,v}(x) &=& 2\sqrt{\Lambda_{a,v}(x)}\, dB_x + 2(1+h\left(\Lambda_{a,v}(x))\right)\, dx \quad \hbox{for } x\in [a,0],\\
\label{edsnegatif}
d\Lambda_{a,v}(x) &=& 2\sqrt{\Lambda_{a,v}(x)}\, dB_x + 2h\left(\Lambda_{a,v}(x)\right)\, dx \quad \hbox{for } x\in [0,\infty),
\end{eqnarray}
where $B$ is a Brownian motion, and \eqref{edsnegatif} holds 
up to the first time, say $w_{a,v}^+$, when it hits $0$, and then is absorbed in $0$ (i.e. $\Lambda_{a,v}(x)=0$ for $x\ge w_{a,v}^+$). Similarly $(\Lambda_{a,v}(a-x),x\ge 0)$ is solution of \eqref{edsnegatif} (with a drift $-2h$ instead of $2h$ and an independent Brownian motion) up to the first time, say $w_{a,v}^-$, when it hits $0$, and then is absorbed in $0$.

For $d\ge 1$, we denote by $\D(\R,\R^d)$ the space of c\`adl\`ag functions $f:\R\to \R^d$ endowed with the usual Skorokhod topology (see for instance Section 12 in \cite{Bil}). The space $\D(\R,\R)$ will also simply be denoted by $\D(\R)$. It will be implicit that all convergences in law of our processes hold in these spaces.

\vspace{0.2cm}
\noindent Our first result is the following theorem:
\begin{theo} 
\label{theo} Assume that $\varphi$ is bounded and Lipschitz. Assume further that for $n$ large enough, $X_{\eps_n}$ is recurrent and that $Y$ is recurrent.  
Then for any finite set $I$, any $a_i\in \R$ and $v_i\ge 0$, $i\in I$,  
$$(\Lambda^{(n)}_{a_i,v_i}(x),x\in \R)_{i\in I}\quad \mathop{\Longrightarrow}_{n\rightarrow\infty}^{\mathcal{L}} \quad \left(\Lambda_{a_i,v_i}(x),x\in \R\right)_{i\in I}.$$
\end{theo}

\noindent As announced above, this theorem gives the convergence of a sequence of excited random walks toward the excited Brownian motion (associated to $\varphi$) at the level of the local times. 
We will also extend this result in a non homogeneous setting, i.e. when $\varphi$ is allowed to depend also on the space variable. We refer the reader to Section \ref{subsecnonh} for more details. 
Actually we will need this extension to prove Theorem \ref{theo} in the case $|I|\ge 2$. This will be explained in Section \ref{Ige2}. 

\vspace{0.2cm}
\noindent Note that if $\varphi$ is compactly supported or nonnegative (and bounded Lipschitz), 
and if $\int_0^\infty \varphi(\l)\, d\l \in (-1,1)$, then $Y$ is recurrent and 
$X_{\eps_n}$ as well for $n$ large enough.  

\vspace{0.2cm}
\noindent A consequence of Theorem \ref{theo} is the following
\begin{cor} 
\label{cor1} Under the hypotheses of Theorem \ref{theo}, for any finite set $I$, any $a_i\in \R$ and $v_i\ge 0$, $i\in I$, 
$$\left(\frac{1}{4n^2}\tau_{\eps_n,[2na_i]}([nv_i])\right)_{i\in I}\quad \mathop{\Longrightarrow}_{n\rightarrow\infty}^{\mathcal{L}} \quad \Big(\tau_{a_i}(v_i)\Big)_{i\in I}.$$ 
\end{cor}
\noindent For $u_i\in \R$, $i\in I$, denote by $\theta_{u_i}^{(i)}$ some independent geometric random variables with parameter $1-e^{-u_i}$, independent of $X_{\eps_n}$. 
Denote also by $\gamma^{(i)}_{u_i}$, $i\in I$, some independent exponential random variables with parameter $u_i$, independent of $Y$. 
Then as in \cite{T1}, we can deduce from the previous results the 
\begin{cor} 
\label{cor2} Under the hypotheses of Theorem \ref{theo}, for any $\lambda_i\ge 0$, $i\in I$, 
$$\left(\frac 1{2n} X_{\eps_n}(\theta^{(i)}_{\lambda_i/(4n^2)})\right)_{i\in I}\quad \mathop{\Longrightarrow}_{n\rightarrow\infty}^{\mathcal{L}} 
\quad \left(Y(\gamma^{(i)}_{\lambda_i})\right)_{i\in I}.$$ 
\end{cor}  
\noindent Finally we get the following:
\begin{theo}
\label{cor3} Assume that $\varphi$ is bounded and Lipschitz. 
For $t\ge 0$, set $X^{(n)}(t):=X_{\eps_n}([4n^2t])/(2n)$, which is well defined at least for $n$ large enough. Then 
$$(X^{(n)}(t),t\ge 0)\quad  \mathop{\Longrightarrow}_{n\rightarrow\infty}^{\mathcal{L}}\quad (Y(t),t\ge 0).$$ 
\end{theo}

\noindent Note that, as opposed to the previous results, we do not assume in this last theorem that $X_{\eps_n}$ and $Y$ are recurrent.

\noindent To obtain this result we need to prove the tightness of the sequence $X_{\eps_n}([4n^2\cdot])/(2n)$, $n\ge 1$. This is done by using a coupling between different branching processes, similar to those which were used for proving Corollary \ref{cor1}. 
The convergence of finite-dimensional distributions follows from Corollary \ref{cor2} and an inversion of Laplace transform. 

\vspace{0.2cm}
\noindent As for Theorem \ref{theo} an extension of this result to the non homogeneous setting can be proved (see Theorem \ref{theofinal} at the end of the paper).  

\vspace{0.2cm}
\noindent Let us mention now a related result of Dolgopyat \cite{D}. He proved a functional central limit theorem for excited random walks when $\varepsilon$ is fixed, and in the recurrent regime;  
more precisely when $\eps_i\ge 0$ for all $i$ and $\alpha:=\sum_i \eps_i<1$. In this case the limiting process is a perturbed Brownian motion, i.e. the process defined by 
$$X_t=B_t+\alpha\left(\sup_{s\le t}X_s-\inf_{s\le t} X_s\right)\quad \textrm{for all }t\ge 0,$$
with $B$ a Brownian motion.

\vspace{0.2cm}
\noindent We will first prove Theorem \ref{theo} in the case $|I|=1$ in Section \ref{sectheo}. 
In section \ref{subsecnonh}, we will extend the result to the non homogeneous setting and in Section \ref{Ige2} we will deduce the result in the general case $|I|\ge 1$. 
Corollaries \ref{cor1} and \ref{cor2} will be proved respectively in Section \ref{seccor} and \ref{seccor2}, and Theorem \ref{cor3} in Section \ref{seccor3}.

\section{Proof of Theorem \ref{theo} in the case $|I|=1$}
\label{sectheo}
We assume in this section that $|I|=1$. Let $a\in \R$ and $v\ge 0$ be given. To fix ideas we assume that $a\le 0$. The case $a\ge 0$ is similar. 
Moreover, we only prove the convergence of $\Lambda^{(n)}_{a,v}$ on the time interval $[a,\infty)$, since the proofs of the convergence on $(-\infty,a]$ and on $[0,+\infty)$ are the same.

\subsection{A criterion of Ethier and Kurtz} 
\label{Sec2.1}
It is now a standard fact and not difficult to check (see however \cite{BaS} or \cite{KZ} for more details) that for all $a\in \N^-$ and $v\in \N$, 
\begin{enumerate}
\item the sequence $(S_{\eps,a,v}(a),\dots,S_{\eps,a,v}(0))$ has the same law as $(V_{\eps,v}(0),\dots,V_{\eps,v}(-a))$, where $(V_{\eps,v}(k),k\ge 0)$ 
is some Markov chain starting from $v$, which is independent of $a$,
\item conditionally to $w=S_{\eps,a,v}(0)$, the sequence $(S_{\eps,a,v}(k),k\ge 0)$ has the same law as some Markov chain $(\widetilde V_{\eps,w}(k),k\ge 0)$, starting from $w$, which is independent of $a$,
\item the sequence $(S_{\eps,a,v}(a-k),k\ge 1)$ has the same law as $(\widetilde V_{-\eps,v+1}(k),k\ge 1)$, where by definition $(-\eps)_i=-\eps_i$ for all $i\ge 1$. 
\end{enumerate}
Moreover, the sequences $(S_{\eps,a,v}(a-k),k\ge 1)$ and $(S_{\eps,a,v}(k),k\ge 0)$ are independent.
The laws of the Markov chains $V_{\eps,v}$ and $\widetilde V_{\eps,v}$ will be described in Subsection \ref{subsecmarkov} in terms of another Markov chain $W_{\eps}$, see in particular \eqref{VW} and \eqref{VW2}. 
Note that this idea to use the Markovian property of the process $S_{\eps,a,v}$ goes back at least to Kesten, Kozlov and Spitzer \cite{KKS}. 

\vspace{0.2cm}
\noindent In the following, in order to lighten the presentation we will forget about the dependence on the starting point (which does not play any serious role here) in the notation for $V_\eps$ and $\widetilde{V}_\eps$. Thus  $V_\eps$ and $V_{\eps_n}$ should be understood respectively as $V_{\eps,v}$ and $V_{\eps_n,[nv]}$, where the $v$ will be clear from the context, and similarly for $\widetilde V_\eps$ and $\widetilde V_{\eps_n}$. 

\vspace{0.2cm} 
\noindent Now we first prove the convergence of $\Lambda^{(n)}_{a,v}$ on $[a,0]$. The proofs of the convergence on $[0,+\infty)$ and on the full interval $[a,+\infty)$
are similar and will be explained in Subsection \ref{secainfini}.  

So on $[a,0]$, $\Lambda^{(n)}_{a,v}$ can be decomposed as a sum of a martingale part 
$M^{(n)}_{a,v}$ and a drift part $B^{(n)}_{a,v}$:  
\begin{equation}\label{eqlambdan}\Lambda^{(n)}_{a,v}(x) = \frac{[nv]}{n}+ M^{(n)}_{a,v}(x) + B^{(n)}_{a,v}(x) \quad \text{for all }x\in [a,0],\end{equation}
with the following equalities in law:
\begin{eqnarray}
\label{Bn}
M^{(n)}_{a,v}(x)= \frac 1 n \sum_{k=1}^{[2nx]-[2na]} \Big\{V_{\varepsilon_n}(k) - \E[V_{\varepsilon_n}(k) \mid V_{\varepsilon_n}(k-1)] \Big\},
\end{eqnarray}
and
\begin{eqnarray}
\label{An}
B^{(n)}_{a,v}(x) = \frac 1 {n} \sum_{k=1}^{[2nx]-[2na]} \Big\{\E[V_{\varepsilon_n}(k) \mid V_{\varepsilon_n}(k-1)]- V_{\varepsilon_n}(k-1) \Big\}.
\end{eqnarray}
Let also $A^{(n)}_{a,v}$ be the previsible compensator of $(M^{(n)}_{a,v})^2$. We have the equality in law: 
\begin{equation} 
\label{A}
A^{(n)}_{a,v}(x) = \frac 1 {n^2} \sum_{k=1}^{[2nx]-[2na]} \Big\{\E[V_{\varepsilon_n}(k)^2 \mid V_{\varepsilon_n}(k-1)]- \E[V_{\varepsilon_n}(k)\mid V_{\varepsilon_n}(k-1)]^2 \Big\},
\end{equation}  
for all $x\in [a,0]$. 

\vspace{0.2cm} 
\noindent We will deduce the convergence of $\Lambda^{(n)}_{a,v}$ from a criterion of Ethier and Kurtz \cite{EK}, namely Theorem 4.1 p.354. According to this result the convergence on $[a,0]$ in Theorem \ref{theo} follows from Propositions \ref{propA} and \ref{propB} below. In addition we need to verify that the martingale problem associated with the operator $2\lambda d^2/(d\lambda)^2+2(1+h(\lambda))d/d\lambda$ is well posed. This follows from Theorem 2.3 p.372 in \cite{EK} (with the notation of \cite{EK} take $r_0=0$ and $r_1=+\infty$). 

\begin{prop}
\label{propA}
Let $R>0$ be given. Set $\tau_n^R:=\inf\{x\ge a\ : \ \Lambda^{(n)}_{a,v}(x)\ge R\}$. Then for $a\le x\le 0\wedge \tau_n^R$, 
$$B^{(n)}_{a,v}(x)= 2\int_a^x (1+h(\Lambda^{(n)}_{a,v}(y)))\, dy + \mathcal{O}\left(\frac{1}{\sqrt{n}}\right),$$
where the $\kO(n^{-1/2})$ is deterministic and only depends on $a$ and $R$. 
\end{prop}

\begin{prop} 
\label{propB}
Let $R>0$ be given. Then for $a\le x\le 0\wedge \tau_n^R$, 
$$A^{(n)}_{a,v}(x) = 4\int_a^x \Lambda^{(n)}_{a,v}(y)\, dy + \mathcal{O}\left(\frac{1}{\sqrt{n}}\right),$$
where the $\kO(n^{-1/2})$ is deterministic and only depends on $a$ and $R$.
\end{prop} 

\noindent These propositions will be proved in the Subsections 2.2--2.5. 


\subsection{An auxiliary Markov chain}
\label{subsecmarkov} 
Let $\varepsilon$ and $v\ge 0$ be given. We express here (see in particular \eqref{VW} and \eqref{VW2} below) the laws of $V_\eps= V_{\varepsilon,v}$ and $\widetilde V_\eps=\widetilde V_{\eps,v}$ in terms of the law of another Markov chain $W_\eps$. 
A similar representation already appeared in T\'oth's paper \cite{T1} on "true" self-avoiding walks. 
So let us first define $(s_{\varepsilon,i},i\ge 0)$ by $s_{\varepsilon,0}=0$ and for $i\ge 1$,  
$$s_{\varepsilon,i}:=\sum_{j=1}^i 1_{\{U_j\ge p_{\eps,j}\}},$$ 
where $(U_j,j\ge 1)$ is a sequence of i.i.d random variables with uniform distribution in $[0,1]$. This $s_{\varepsilon,i}$ is equal in law to the number of times the excited random walk jumps from level $k$ to $k-1$, for some arbitrary $k\in \Z$, after $i$ visits at this level $k$. 
For $m\ge 0$, set 
$$W_\varepsilon(m):=\inf\{i\ge 0\ :\ s_{\varepsilon,i}=m\}.$$ 
Then $W_\varepsilon(m)$ is equal in law to the number of visits to level $k$ before the $m$-th jump from $k$ to $k-1$. 
Moreover, $(W_\varepsilon(m),m\ge 0)$ is a Markov chain on $\N$ starting from $0$ and 
with transition operator $Q_\varepsilon$ defined for any nonnegative or bounded function $f$ by 
$$Q_\varepsilon f(r) = \sum_{\l\ge 1} f(r+\l)2^{-\l} (1+\varepsilon_{r+1})\dots(1+\varepsilon_{r+\l-1})(1-\varepsilon_{r+\l}),$$
for all $r\in \N$. Furthermore it is immediate that the law of $V_{\varepsilon}(k+1)$ conditionally on $\{V_{\varepsilon}(k)=m\}$ is equal to the law of $W_{\varepsilon}(m)-m+1$: 
\begin{eqnarray}
\label{VW}
\kL(V_{\varepsilon}(k+1) \mid V_{\varepsilon}(k)=m) = \kL(W_{\varepsilon}(m)-m+1).
\end{eqnarray}
Similarly the law of $\widetilde V_{\varepsilon}(k+1)$ conditionally on $\{\widetilde V_{\varepsilon}(k)=m\}$ is equal to the law of $W_{\varepsilon}(m)-m$:
\begin{eqnarray}
\label{VW2}
\kL(\widetilde V_{\varepsilon}(k+1) \mid \widetilde V_{\varepsilon}(k)=m) = \kL(W_{\varepsilon}(m)-m).
\end{eqnarray}
By convention we denote by $Q_0$ the transition operator associated to the sequence $(\varepsilon_i,i\ge 1)$, where $\varepsilon_i=0$ for all $i$. In other words
$$Q_0f(r)=\E[f(r+\xi)]\quad \text{for all }r\in \N,$$
where $\xi$ is a geometric random variable with parameter $1/2$, i.e. $\pp(\xi=\l)=2^{-\l}$, for all $\l\ge 1$. Note that $\E(\xi)=2$ and $\V(\xi)= 2$. 
In particular, if $u$ is defined by $u(r)=r$ for all $r\in \N$, then for all $m\ge 1$, 
$$Q_0^m u(0)=\E[\xi_1+\dots+\xi_m]=2m,$$
where $\xi_1,\dots,\xi_m$ are i.i.d. geometric random variables with parameter $1/2$.
Note also that for all $m\ge 1$, $\E[W_{\varepsilon}(m)] = Q_{\varepsilon}^m u (0)$. Thus \eqref{VW} shows that 
\begin{equation}
\label{VkQk}
\E[V_{\eps}(k)\mid V_{\eps}(k-1)]-V_{\eps}(k-1) = Q_{\epsilon}^{V_{\eps}(k-1)}u(0)-Q_0^{V_{\eps}(k-1)}u(0)+1,
\end{equation}
for all $k\ge 1$. So in view of \eqref{An} and \eqref{VkQk}, our strategy for proving Proposition \ref{propA} will be to estimate terms of the form $Q_{\varepsilon}^m u(0) -Q_0^m u(0)$. Note that since $x<\tau_n^R$ by hypothesis, we can restrict us to the case when $m\le Rn+1$. Likewise 
\begin{equation}
\label{VkQkbis}
\E[V_{\eps}^2(k) \mid V_{\varepsilon}(k-1)]- \E[V_{\varepsilon}(k)\mid V_{\varepsilon}(k-1)]^2= Q_{\eps}^{V_{\eps}(k-1)}u^2(0) - (Q_{\eps}^{V_{\eps }(k-1)}u(0))^2,
\end{equation}
for all $k\ge 1$. So in view of \eqref{A} and \eqref{VkQkbis} we will have also to estimate terms of the form $Q_{\varepsilon}^m u^2(0) -(Q_{\eps}^m u(0))^2$, for proving Proposition \ref{propB}.   

\subsection{Some elementary properties of the operators $Q_\varepsilon$ and $Q_0$}
For $f:\N\to \R$, we set 
\begin{eqnarray*}
|f|_\infty &=& \sup_{r\in \N} |f(r)|, \\
\textrm{Lip}(f) &=& \sup_{r\neq r'} \frac{ |f(r)-f(r')|}{|r-r'|},
\end{eqnarray*}
and
\begin{eqnarray*}
\textrm{Lip}_2(f) &=& \sup_{\l\in\N} \textrm{Lip}(\Delta_\l f),
\end{eqnarray*}
where $\Delta_\l f(r)=f(r+l)-f(r)$. Naturally we say that $f$ is Lipschitz if $\textrm{Lip}(f)<+\infty$. 
Note that for any $f$, $\textrm{Lip}_2(f)\le 2\textrm{Lip}(f)\le 4|f|_\infty$. 
Set 
$$\kD:=\left\{h:\N\to \R\ :\ \textrm{there exists }f,\ g \textrm{ Lipschitz, such that } h(r)=f(r)+rg(r)\right\}.$$

\noindent For $\eps : \N\to (-1,1)$, 
let $R_\varepsilon := Q_\varepsilon - Q_0$. Note that when $|\eps|_\infty\le 1/2$, $R_\eps h$ and $Q_\eps h$ are 
well defined for any $h \in \kD$. 
Observe also that $Q_\varepsilon 1 =Q_0 1 = 1$, where $1$ is the constant function on $\N$. In particular $R_\varepsilon 1=0$. 
Moreover, for any Lipschitz $f$, $|Q_\varepsilon f|_\infty \le |f|_\infty$, and 
\begin{eqnarray}
\label{QepsLip}
|Q_\varepsilon f-f|_\infty \le C \textrm{Lip}(f),
\end{eqnarray}
where $C=\sum_{\l\ge 1} \l(4/3)^{-\l}$. As a corollary we get the 
\begin{lem} 
\label{Qepsf-f}
There exists a constant $C>0$ such that for all $|\eps|_\infty \le 1/2$, all $j\ge 0$ and all Lipschitz functions $f$,  
$$|Q_\eps^j f - f|_\infty \le  C j \textrm{Lip}(f).$$
\end{lem} 
\begin{proof} 
Write 
$$Q_\eps^j f - f = \sum_{i=1}^j Q_\eps^{i-1}(Q_\eps f-f),$$
and then use \eqref{QepsLip} for each term of the sum. 
\end{proof} 
\noindent Set for all $r\ge 0$ and $\l\ge 1$,  
$$\widetilde{\varepsilon}_{r,\l}:= -\varepsilon_{r+\l}+\sum_{i=1}^{\l-1} \varepsilon_{r+i},$$
and define $\widetilde{R}_{\varepsilon}$ by 
$$\widetilde{R}_{\varepsilon} f (r) = \sum_{\l \ge 1} f(r+\l)2^{-\l}\widetilde{\varepsilon}_{r,\l}.$$ 
This $\widetilde{R}_{\varepsilon}$ is a linearized version of $R_{\varepsilon}=Q_\eps-Q_0$, and also the first order term in the expansion of $R_\eps$ as $|\eps|_\infty \to 0$. The next result is immediate. 
\begin{lem} \label{lem1} Assume that $|\eps|_\infty\le 1/2$. Then for any $h\in \kD$ and any $r$, 
$$\widetilde{R}_\varepsilon h(r)=-\sum_{\l\ge 1}\varepsilon_{r+\l}\left(h(r+\l)2^{-\l}-\sum_{i=\l+1}^\infty h(r+i)2^{-i}\right).$$
\end{lem}
\noindent In particular $\widetilde{R}_\varepsilon 1=0$ since $\sum_{i=\l+1}^\infty 2^{-i}=2^{-\l}$. We also get the following
\begin{lem} 
\label{maj2}
There exists a constant $C>0$ such that for all $|\varepsilon|_\infty\le 1/2$ and all $h\in \kD$, with $h(r)=f(r)+rg(r)$,
\begin{equation}
\label{eq:repsh} 
\widetilde{R}_\varepsilon  h (r) = f_\varepsilon (r) + rg_\epsilon (r),
\end{equation}
where $f_\eps$ and $g_\eps$ satisfy\\
\begin{enumerate}
\item[(i)] \quad $|f_\eps|_\infty \le C(\textrm{Lip}(f)+|g|_\infty)\times |\eps|_\infty$ 
\item[(ii)] \quad$|g_\eps|_\infty \le C\textrm{Lip}(g)\times |\eps|_\infty$
\item[(iii)] \quad   $\textrm{Lip}(f_\eps) \le C (\textrm{Lip}_2(f)+\textrm{Lip}(g))\times |\eps|_\infty + C (\textrm{Lip}(f)+|g|_\infty) \times \textrm{Lip}(\eps)$
\item[(iv)] \quad$ \textrm{Lip}(g_\eps) \le C\textrm{Lip}_2(g)\times |\eps|_\infty + C\textrm{Lip}(g)\times \textrm{Lip}(\eps)$.
\end{enumerate}
\end{lem} 
\begin{proof} 
By using that $h(r)\widetilde{R}_\eps 1(r)=0$ for all $r$, we get
$$\widetilde{R}_\eps h(r)=-\sum_{\l\ge 1}\eps_{r+\l}\left((h(r+\l)-h(r))2^{-\l}-\sum_{i=\l+1}^\infty
(h(r+i)-h(r))2^{-i}\right).$$
Thus \eqref{eq:repsh} holds with
\begin{eqnarray*}
f_\eps(r) &=& -\sum_{\l\ge 1}\eps_{r+\l}\left((f(r+l)-f(r))2^{-\l}-\sum_{i=\l+1}^\infty
(f(r+i)-f(r))2^{-i}\right)\\
&&  -\sum_{\l\ge 1}\eps_{r+\l}\left(g(r+\l)\l2^{-\l}-\sum_{i=\l+1}^\infty
g(r+i)i2^{-i}\right),
\end{eqnarray*}
and
\begin{eqnarray*} g_\eps(r) &=& -\sum_{\l\ge 1}\eps_{r+\l}\left((g(r+\l)-g(r))2^{-\l}-\sum_{i=\l+1}^\infty
(g(r+i)-g(r))2^{-i}\right).
\end{eqnarray*}
All assertions follow immediately. For instance we can write 
$$|g_\eps|_\infty \le \textrm{Lip}(g)\times|\eps|_\infty \times\sum_{\l\ge 1}\left(\l2^{-\l}+\sum_{i=\l+1}^\infty
i2^{-i}\right),$$
which implies (ii) and one can prove similarly (i), (iii) and (iv).
\end{proof} 
\noindent Next we have
\begin{lem} 
\label{Rtilde}
There exists a constant $C>0$ such that for all $|\varepsilon|_\infty\le 1/2$ and all $h\in \kD$, with $h(r)=f(r)+rg(r)$,
$$R_\varepsilon h(r)-\widetilde{R}_\varepsilon h(r)=f_\eps(r)+rg_\eps(r),$$ 
where $f_\eps$ and $g_\eps$ satisfy
\begin{eqnarray*} 
|f_\eps|_\infty &\le& C(\textrm{Lip}(f)+|g|_\infty)\times |\varepsilon|_\infty^2,\\
|g_\eps|_\infty &\le& C\textrm{Lip}(g)\times |\varepsilon|_\infty^2.
\end{eqnarray*}
\end{lem}
\begin{proof} Recall that 
$$(R_\varepsilon  - \widetilde{R}_\varepsilon)h(r)= \sum_{\l\ge 1}h(r+\l)2^{-\l}(\eps_{r,\l}-\widetilde{\eps}_{r,\l}),$$
where for all $r$ and $\l$, 
$$\eps_{r,\l} := (1+\varepsilon_{r+1})\dots(1+\varepsilon_{r+\l-1})(1-\varepsilon_{r+\l})-1.$$
Since $h(r)R_\eps 1 (r) = h(r) \widetilde{R}_\eps 1(r) = 0$ for all $r$, we get 
\begin{eqnarray*}(R_\varepsilon  - \widetilde{R}_\varepsilon)h(r)&=& \sum_{\l\ge 1}(h(r+\l)-h(r))2^{-\l}(\eps_{r,\l}-\widetilde{\eps}_{r,\l})\\
&=& f_\eps(r)+rg_\eps(r).
\end{eqnarray*}
with
\begin{eqnarray*}
f_\eps(r) &=& \sum_{\l\ge 1}(f(r+l)-f(r))2^{-\l} (\eps_{r,\l}-\widetilde{\eps}_{r,\l}) \\
&& + \sum_{\l\ge 1} g(r+\l)\l 2^{-\l} (\eps_{r,\l}-\widetilde{\eps}_{r,\l}),
\end{eqnarray*}
and
\begin{eqnarray*} g_\eps(r) &=& \sum_{\l\ge 1}(g(r+\l)-g(r))2^{-\l} (\eps_{r,\l}-\widetilde{\eps}_{r,\l}).
\end{eqnarray*}
But for any $r$ and any $\l\ge 1$,  
\begin{eqnarray}
\label{epstildeeps} 
|\varepsilon_{r,\l}-\widetilde{\varepsilon}_{r,\l}|&\le & \nonumber (1+|\varepsilon|_\infty)^{\l}-1-\l|\varepsilon|_\infty\\
                                                   &\le & \l^2 (1+|\varepsilon|_\infty)^{\l-2} |\varepsilon|_\infty^2 \\
 			\nonumber 			&\le & 	\l^2 (3/2)^{\l-2} |\eps|_\infty^2.
\end{eqnarray}
The lemma follows.  
\end{proof} 
\noindent Lemmas \ref{maj2} and \ref{Rtilde} imply
\begin{lem} \label{lemR} There exists a constant $C>0$ such that for all $|\varepsilon|_\infty\le 1/2$ and all $h\in \kD$, with $h(r)=f(r)+rg(r)$, 
$$R_\varepsilon h(r)=f_\eps(r)+rg_\eps(r),$$ 
where $f_\eps$ and $g_\eps$ satisfy
\begin{eqnarray*} 
|f_\eps|_\infty &\leq& C(\textrm{Lip}(f)+|g|_\infty)\times |\varepsilon|_\infty,\\
|g_\eps|_\infty &\le& C\textrm{Lip}(g)\times |\varepsilon|_\infty.
\end{eqnarray*} 

\end{lem}

\noindent We will need also the following
\begin{lem} 
\label{Q0i} For all $h\in \kD$, with $h(r)=f(r)+rg(r)$, and all $i\ge 0$, 
$$Q_0^ih(r)=f_i(r)+rg_i(r),$$ 
where $f_i$ and $g_i$ satisfy
\begin{enumerate}
\item[(i)] \quad $|f_i|_\infty\le |f|_\infty+2i|g|_\infty$
\item[(ii)] \quad $\textrm{Lip}(f_i)\le \textrm{Lip}(f)+2i \textrm{Lip}(g)$
\item[(iii)] \quad $|g_i|_\infty\le |g|_\infty$
\item[(iv)] \quad $\textrm{Lip}(g_i)\le \textrm{Lip}(g)$.
\end{enumerate}
Moreover, for all $r$,
$$|Q_0^ih(r)-h(r+2i)|\le \sqrt{2i} (\textrm{Lip}(f)+|g|_\infty + (r+2i)\textrm{Lip}(g)).$$
\end{lem}
\begin{proof} 
Just recall that for all $i$ and $r$, $Q_0^ih(r)=\E[h(r+\xi_1+\dots+\xi_i)]$, where $\xi_1,\dots,\xi_i$ are i.i.d. geometric random variables with parameter $1/2$. Thus, $Q_0^ih(r)=f_i(r)+rg_i(r)$, where
$$f_i(r)=\E[f(r+\xi_1+\dots+\xi_i)] + \E[(\xi_1+\dots+\xi_i)g(r+\xi_1+\dots+\xi_i)],$$
and 
$$g_i(r)=\E[g(r+\xi_1+\dots+\xi_i)].$$
Claims (i), (ii), (iii) and (iv) follow immediately frome these expressions (by using also that $\E[|\xi_1+\dots+\xi_i|]=2i$). Next write 
\begin{eqnarray*}
|Q_0^if(r)-f(r+2i)| &\le & \E[|f(r+\xi_1+\dots+\xi_i)-f(r+2i)|]\\
					&\le & \textrm{Lip}(f) \E[|\xi_1+\dots+\xi_i-2i|]\\
					&\le & \sqrt{2i} \textrm{Lip}(f),
\end{eqnarray*}
by using Cauchy-Schwarz inequality and the fact that $\E(\xi_i)=2$ and $\V(\xi_i)=2$, for all $i$. We also have
\begin{eqnarray*}
|Q_0^i(h-f)(r)-(r+2i)g(r+2i)| 
&\le &   \E[|(\xi_1+\dots+\xi_i-2i)g(r+\xi_1+\dots+\xi_i)|\\
&& + \E[|(r+2i)(g(r+\xi_1+\dots+\xi_i)-g(r+2i))|]\\
&\le & \sqrt{2i}\left(|g|_\infty + (r+2i) \textrm{Lip}(g)\right).
\end{eqnarray*}
\end{proof}

\begin{lem}
\label{Qepsi}
There exists a constant $C>0$ such that for all $|\eps|_\infty\le 1/2$, all $i\ge 0$ and all Lipschitz $f$, 
$$|(Q_\eps^i-Q_0^i)f|_\infty \le C i \hbox{Lip}(f)\times |\eps|_\infty.$$
\end{lem}
\begin{proof}
First write 
\begin{eqnarray}\label{qepsiq0} 
Q_\eps^i-Q_0^i = \sum_{j=0}^{i-1} Q_\eps^{i-j+1} R_\eps Q_0^{j}.
\end{eqnarray} 
Then by using that $|Q_\eps^{i-j+1} f|_\infty \le |f|_\infty$, for all $j\le i-1$, we get (using Lemma \ref{lemR} with $g=0$),
\begin{eqnarray*}
|(Q_\eps^i-Q_0^i)f|_\infty \le  \sum_{j=0}^{i-1} |R_\eps Q_0^j f|_\infty \le  C \sum_{j=0}^{i-1} \hbox{Lip}(Q_0^j f) |\eps |_\infty.
\end{eqnarray*} 
We conclude the proof of the lemma by using that $\hbox{Lip}(Q_0^j f) \le \hbox{Lip}(f)$ for all $j$. 
\end{proof}

\begin{lem} \label{NL} There exists a constant $C>0$ such that for all $|\eps|_\infty\le 1/2$, all $i\ge 0$ 
and all $h \in \kD$, with $h(r)=f(r)+rg(r)$, 
$$|Q_\eps^ih(r)|\le |f|_\infty + (r+Ci)|g|_\infty . $$
\end{lem}
\begin{proof} We have
\begin{eqnarray*}
|Q_\eps^ih(r)| \le |f|_\infty + |g|_\infty |Q_\eps^iu(r)|,
\end{eqnarray*}
where $u$ is defined by $u(r)=r$ for all $r\in \N$. 
Now, Lemma \ref{Qepsi} implies that 
$$|Q_\eps^iu(r) - Q_0^iu(r)|\le  Ci|\eps|_\infty.$$
We conclude by using that $Q_0^iu(r)=r+2i$.
\end{proof}

\begin{lem}\label{Qepsibis}
There exists a constant $C>0$ such that for all $|\eps|_\infty\le 1/2$, all $i\ge 0$ and all $h\in \kD$, 
with $h(r)=f(r)+rg(r)$, 
$$|(Q_\eps^i-Q_0^i)h(r)|\le C i \left(\hbox{Lip}(f)+|g|_\infty +(r+i)\hbox{Lip}(g)\right)\times |\eps|_\infty.$$
\end{lem}
\begin{proof}
We have $Q_0^jh(r)=f_j(r)+rg_j(r)$ and $R_\eps Q_0^jh(r)=f_{j,\eps}(r)+rg_{j,\eps}(r)$. Lemma \ref{NL} implies that $$|Q_\eps^{i-j+1} R_\eps Q_0^{j}h(r)|\le |f_{j,\eps}|_\infty + (r+C(i-j+1))|g_{j,\eps}|_\infty.$$
Lemma \ref{lemR} implies that 
\begin{eqnarray*} 
|f_{j,\eps}|_\infty\leq C(\textrm{Lip}(f_j)+|g_j|_\infty)\times |\varepsilon|_\infty \hbox{ and } |g_{j,\eps}|_\infty\le C\textrm{Lip}(g_j)\times |\varepsilon|_\infty.
\end{eqnarray*} 
Lemma \ref{Q0i} implies that 
$$|f_{j,\eps}|_\infty\leq C(\textrm{Lip}(f)+|g|_\infty +2j \textrm{Lip}(g))\times |\eps|_\infty,$$
and 
$$|g_{j,\eps}|_\infty\le C\textrm{Lip}(g)\times |\varepsilon|_\infty.$$
Now,
\begin{eqnarray*}
\sum_{j=0}^{i-1} |f_{j,\eps}|_\infty &\le& C i (\textrm{Lip}(f)+|g|_\infty +i \textrm{Lip}(g))\times |\eps|_\infty,
\end{eqnarray*}
and
\begin{eqnarray*}
\sum_{j=0}^{i-1} (r+C(i-j+1))|g_{j,\eps}|_\infty &\le& C(r+i)i \times \textrm{Lip}(g)\times |\varepsilon|_\infty. \end{eqnarray*}
Using then \eqref{qepsiq0}, this proves the lemma.
\end{proof}

\noindent Our last result in this subsection is the following (recall that $u(r)=r$ for all $r\in \N$):
\begin{lem}
\label{Qepsm}
There exists a constant $C>0$ such that for all $m\ge 0$ and all $|\eps|_\infty\le 1/2$, 
$$|Q_\eps^m u-Q_0^m u-\sum_{i=1}^m Q_0^{m-i}\widetilde{R}_\eps u|_\infty \le C (m|\eps|_\infty^2 + m^2 |\eps|_\infty \hbox{Lip}(\eps)).$$
\end{lem} 
\begin{proof}
First observe that for all $j\ge 0$, $Q_0^j u = u+2j$. Since $\widetilde{R}_\eps$ is linear and $\widetilde{R}_\eps 1=0$, we get $\widetilde{R}_\eps Q_0^{i-1}u= \widetilde{R}_\eps u$ for all $i\ge 1$. Thus 
$$\sum_{i=1}^m Q_0^{m-i}\widetilde{R}_\eps u=\sum_{i=1}^m Q_0^{m-i}\widetilde{R}_\eps Q_0^{i-1}u.$$
Next we have 
\begin{eqnarray*}
Q_\eps^m u-Q_0^m u- \sum_{i=1}^m Q_0^{m-i}\widetilde{R}_\eps Q_0^{i-1}u &= & \sum_{i=1}^m Q_0^{m-i}(R_\eps- \widetilde{R}_\eps) Q_0^{i-1}u \\
																	& & + \sum_{i=1}^m (Q_\eps^{m-i}-Q_0^{m-i})\widetilde{R}_\eps Q_0^{i-1}u.
\end{eqnarray*} 
By using Lemma \ref{Rtilde} and the fact that $\hbox{Lip}(Q_0^iu)\le \hbox{Lip}(u)=1$ for all $i$, we get 
\begin{eqnarray*}
\sum_{i=1}^m |Q_0^{m-i}(R_\eps- \widetilde{R}_\eps) Q_0^{i-1}u|_\infty &\le & C \sum_{i=1}^m \hbox{Lip}(Q_0^{i-1}u) |\eps|_\infty^2 \\
																		&\le & Cm|\eps|_\infty^2.
\end{eqnarray*}
Then by using Lemma \ref{Qepsi} we obtain 
$$\sum_{i=1}^m |(Q_\eps^{m-i}-Q_0^{m-i})\widetilde{R}_\eps Q_0^{i-1}u|_\infty \le C|\eps|_\infty \sum_{i=1}^m (m-i) \hbox{Lip}(\widetilde{R}_\eps u).$$
Using Lemma \ref{maj2} with $f(r)=1$ and $g(r)=0$, we have $g_\eps(r)=0$ and 
\begin{equation}\label{eq:reps} \hbox{Lip}(\widetilde{R}_\eps u)=\hbox{Lip}(f_\eps) \le C \hbox{Lip}(\eps). \end{equation} 
This proves the lemma.
\end{proof}


\subsection{Proof of Proposition \ref{propA}}
Recall that $\varepsilon_n=(\varepsilon_{i}(n),i\ge 1)$, with $\varepsilon_{i}(n)=  \varphi(i/2n)/(2n)$. Since $\varphi$ is bounded, we can always assume by taking large enough $n$ if necessary, that $|\eps_n|_\infty \le 1/2$. Note also that $\textrm{Lip}(\eps_n) = \kO(1/n^2)$. Assume now that $m=\kO(n)$. Then Lemma \ref{Qepsm} shows that 
$$Q_{\eps_n}^m u-Q_0^m u= \sum_{i=0}^{m-1} Q_0^i\widetilde{R}_{\eps_n} u + \kO\left(\frac 1 n\right).$$
Next write
\begin{eqnarray*}
\sum_{i=0}^{m-1} Q_0^i\widetilde{R}_{\eps_n} u(0) = \sum_{i=0}^{m-1} \widetilde{R}_{\eps_n} u(2i) + \sum_{i=0}^{m-1} (Q_0^i\widetilde{R}_{\eps_n} u(0)-\widetilde{R}_{\eps_n} u(2i)).
\end{eqnarray*}
By using Lemma \ref{Q0i} (applied to $f=\widetilde{R}_{\eps_n} u$ and $g=0$) we get 
\begin{eqnarray*}
\sum_{i=0}^{m-1} |Q_0^i\widetilde{R}_{\eps_n} u(0)-\widetilde{R}_{\eps_n} u(2i)|&\le & \sqrt{2} 	\sum_{i=0}^{m-1} \hbox{Lip}(\widetilde{R}_{\eps_n}u)\sqrt{i}\\
																				&\le & C m^{3/2} \hbox{Lip}(\eps_n) \\
																				&\le & \frac C {\sqrt{n}}.
\end{eqnarray*}
On the other hand, set 
\begin{eqnarray}
\label{al}
a_\l:= -\l 2^{-\l}+\sum_{j=\l+1}^\infty  j 2^{-j} = 2^{-\l+1}.
\end{eqnarray}
Then by using Lemma \ref{lem1} we get 
\begin{eqnarray*}
\sum_{i=0}^{m-1}\widetilde{R}_{\eps_n}u(2i) &=& \sum_{i=0}^{m-1} \sum_{\l=1}^\infty
a_\l (\eps_n)_{2i+\l}\\
&=& \sum_{\l=1}^\infty 2^{-\l+1}\times\frac{1}{2n}\sum_{i=0}^{m-1} \varphi\left(\frac{2i+\l}{2n}\right).
\end{eqnarray*}
But $\sum_{\l=1}^\infty 2^{-\l+1}=2$, and since $\varphi$ is Lipschitz and bounded
$$\frac{1}{n}\sum_{i=0}^{m-1} \varphi\left(\frac{2i+\l}{2n}\right)=\int_0^{m/n} \varphi(s)\, ds +\kO\left(\frac \l n\right).$$
Thus putting the pieces together we get
\begin{eqnarray}
\label{final1}
Q_{\eps_n}^m u(0)-Q_0^m u(0)=  h\left(\frac m n\right) +\kO\left(\frac{1}{\sqrt{n}}\right).
\end{eqnarray}
Finally we get the equalities in law, for $a\le x\le 0\wedge \tau_n^R$, 
\begin{eqnarray*}
B^{(n)}_{a,v}(x) &=& \frac 1 {n} \sum_{k=1}^{[2nx]-[2na]}  \left(\E[V_{\varepsilon_n}(k) \mid V_{\varepsilon_n}(k-1)]- V_{\varepsilon_n}(k-1) \right) \\
          &=& \frac 1 {n} \sum_{k=1}^{[2nx]-[2na]} \left(1+ Q^{V_{\varepsilon_n}(k-1)}_{\varepsilon_n}u(0)-Q_0^{V_{\varepsilon_n}(k-1)}u(0)\right)\\
          &=& \frac 1 {n} \sum_{k=1}^{[2nx]-[2na]} \left\{1 + h\left(\Lambda^{(n)}_{a,v}\left(a+\frac{k-1}{2n} \right)\right)\right\} + \kO\left(\frac{1}{\sqrt{n}}\right)\\
          &=& 2\int_a^x \left\{1 + h(\Lambda^{(n)}_{a,v}(y))\right\}\, dy + \kO\left(\frac{1}{\sqrt{n}}\right),
\end{eqnarray*}
where the second equality follows from \eqref{VkQk} and the third one from \eqref{final1} and the relation between $\Lambda^{(n)}_{a,v}$ and $V_{\eps_n}$ 
given in \eqref{LambdaS} and at the begining of Subsection \ref{Sec2.1}. This finishes the proof of Proposition \ref{propA}. \hfill $\square$


\subsection{Proof of Proposition \ref{propB}}
We assume throughout this subsection that $m=\kO(n)$. Then on the one hand by using Lemma \ref{Qepsm}, we get 
$$Q_{\eps_n}^m u(0)= 2m + \sum_{i=1}^{m} Q_0^{m-i}\widetilde R_{\eps_n}u(0) + \kO\left(\frac 1 n\right).$$ 
Moreover Lemma \ref{maj2} shows that $|Q_0^{i}\widetilde R_{\eps_n}u(0)|\le |\widetilde R_{\eps_n}u| \le C|\eps_n|_\infty=\kO(1/n)$ uniformly in $i$. Thus 
$$(Q_{\eps_n}^m u(0))^2= 4m^2 + 4m \sum_{i=1}^{m} Q_0^{m-i}\widetilde R_{\eps_n}u(0) + \kO(1).$$ 
On the other hand we have for all $\eps$,
$$Q_{\eps}^m u^2 = Q_0^m u^2 + \sum_{i=1}^m Q_{\eps}^{m-i}R_{\eps}Q_0^{i-1}u^2.$$
A variance calculus shows that 
$$Q_0^{i-1}u^2 = u^2+4(i-1) u + 4(i-1)^2 + 2(i-1),$$
which implies that for all $\eps$,
$$R_\eps Q_0^{i-1} u^2 = R_\eps u^2 +4(i-1) R_\eps u,$$
since $R_\eps 1=0$. Thus 
$$Q_\eps^mu^2(0)= 4m^2+2m+E_{\eps,m}+F_{\eps,m},$$
where 
$$E_{\eps,m}=\sum_{i=1}^m Q_\eps^{m-i}R_\eps u^2(0),$$
and 
$$F_{\eps,m}=4\sum_{i=1}^m(i-1)Q_\eps^{m-i}R_\eps u(0).$$
We now prove the following
\begin{lem} We have
\begin{eqnarray} 
\label{E}
E_{\eps_n,m} = \sum_{i=1}^mQ_0^{m-i}\widetilde R_{\eps_n}u^2(0) + \kO(1),
\end{eqnarray}
and 
\begin{eqnarray}
\label{F}
F_{\eps_n,m}=4\sum_{i=1}^m(i-1)Q_0^{m-i}\widetilde R_{\eps_n}u(0)+\kO(1).
\end{eqnarray}
\end{lem}
\begin{proof} 
We have 
\begin{eqnarray*}
(R_\eps-\widetilde R_\eps) u^2(r) = \sum_{\l\ge 1}(2r\l + \l^2)2^{-\l}(\eps_{r,\l}-\widetilde \eps_{r,\l}) \quad \textrm{for all } r.
\end{eqnarray*}
Thus, by using \eqref{epstildeeps}, we see that there exists a constant $C>0$ such that
\begin{eqnarray}
\label{RRtildebis}
|(R_{\eps_n}-\widetilde R_{\eps_n}) u^2(r)|\le C|\eps_n|^2_\infty(r+1)\le C \frac{(r+1)}{n^2} \quad \textrm{for all } r.
\end{eqnarray}
This implies that
\begin{eqnarray*}
|E_{\eps_n,m}-\sum_{i=1}^m Q_{\eps_n}^{m-i}\widetilde R_{\eps_n}u^2(0)| &\le & 
\sum_{i=1}^m Q_{\eps_n}^{m-i}|R_{\eps_n}u^2-\widetilde R_{\eps_n}u^2|(0) \\
&\le & \frac C {n^2} \sum_{i=1}^m Q_{\eps_n}^{m-i}f(0)
\end{eqnarray*} 
with $f(r)=r+1$.
By using Lemma \ref{Qepsf-f}, applied to $f(r)=r+1$, we see that there exists $C>0$ such that 
\begin{eqnarray*}
|E_{\eps_n,m}-\sum_{i=1}^m Q_{\eps_n}^{m-i}\widetilde R_{\eps_n}u^2(0)| 
&\le & \frac C {n^2} \sum_{i=1}^m (1+m-i)\\
&\le & C \frac{m^2}{n^2} = \kO(1).
\end{eqnarray*}
Recall the formula for $a_\l$ given in \eqref{al} and let 
$$b_\l : = -\l^2 2^{-\l}+\sum_{i=\l+1}^\infty i^2 2^{-i}.$$
Then Lemma \ref{lem1} shows that 
\begin{equation}
\label{repsu2}
\widetilde R_{\eps_n} u^2(r)=f_n(r)+rg_n(r),
\end{equation} 
where
\begin{eqnarray*}
f_n(r) &=& \sum_{\l \ge 1} b_\l (\eps_n)_{r+\l}\\
g_n(r) &=& 2\sum_{\l \ge 1} a_\l (\eps_n)_{r+\l}.
\end{eqnarray*}
Next by using Lemma \ref{Qepsibis}, we get for all $j=m-i$ and $1\le i\le m$, 
\begin{eqnarray*}
\label{QepsQ0mi}
|(Q_{\eps_n}^{j}-Q_0^{j})\widetilde R_{\eps_n}u^2(0)| &\le&
C j \left(\hbox{Lip}(f_n)+|g_n|_\infty +j\hbox{Lip}(g_n)\right)\times |\eps_n|_\infty\\
&\le& Cj\left(|\eps_n|_\infty +(j+1)\hbox{Lip}(\eps_n)\right)\times |\eps_n|_\infty\\
&\le& C\left(\frac{m}{n^2}+\frac{m^2}{n^3}\right) = \kO\left(\frac{1}{n}\right).
\end{eqnarray*}
This proves \eqref{E}.

\noindent Now Lemmas \ref{Rtilde} and \ref{Qepsi}, together with \eqref{eq:reps}, show that (for $j=m-i$)
\begin{eqnarray*}
|Q_{\eps_n}^{j}R_{\eps_n} u - Q_0^{j}\widetilde R_{\eps_n} u |_\infty 
&\le & |Q_{\eps_n}^{j}(R_{\eps_n} -\widetilde R_{\eps_n}) u|_\infty  + |(Q_{\eps_n}^{j}-Q_0^{j})\widetilde R_{\eps_n} u |_\infty \\ 
& = & \kO\left(|\eps_n|_\infty^2 + j|\eps_n|_\infty \textrm{Lip}(\eps_n)\right) =\kO\left(\frac 1 {n^2}\right).
\end{eqnarray*}
This proves \eqref{F} and finishes the proof of the lemma. 
\end{proof}
\noindent We can now write 
\begin{eqnarray*}
Q_{\eps_n}^mu^2(0)-(Q_{\eps_n}^mu(0))^2 &=&  2m + \sum_{j=0}^{m-1} Q_0^j \widetilde R_{\eps_n}u^2(0) \\
   & & + 4\sum_{j=0}^{m-1} (m-j-1) Q_0^j \widetilde{R}_{\eps_n}u(0)\\
   & & -4m \sum_{j=0}^{m-1} Q_0^j \widetilde R_{\eps_n} u(0) + \kO(1).
\end{eqnarray*}
By using Lemmas \ref{Q0i} and the form of $\widetilde R_{\eps_n}u^2$ given by \eqref{repsu2} (and using that $\hbox{Lip}(f_n)=\kO(n^{-2})$, $\hbox{Lip}(g_n)=\kO(n^{-2})$ and $|g_n|_\infty=\kO(n^{-1})$) we get for $j\le m-1$, 
$$Q_0^j\widetilde R_{\eps_n}u^2(0) = \widetilde R_{\eps_n}u^2(2j) + \kO(n^{-1/2}).$$
By using Lemmas \ref{Q0i}, the fact that $\widetilde R_{\eps_n}$ is Lipschitz and bounded, and \eqref{eq:reps}, we get
$$Q_0^j \widetilde{R}_{\eps_n} u(0)= \widetilde R_{\eps_n}u(2j) + \kO(n^{-3/2}).$$
Therefore
\begin{eqnarray*}
Q_{\eps_n}^mu^2(0)-(Q_{\eps_n}^mu(0))^2 &=&  2m + \sum_{j=0}^{m-1}  \widetilde R_{\eps_n}u^2(2j) \\
   & & - 4\sum_{j=0}^{m-1} (j+1) \widetilde{R}_{\eps_n}u(2j)+ \kO(n^{1/2}).
\end{eqnarray*}
Lemma \ref{lem1} shows that 
$$\widetilde R_{\eps_n}u^2(2j) = 4j \sum_{\l \ge 1} a_\l  (\eps_n)_{2j+\l}  + \kO\left(\frac 1 n\right),$$
and
$$\widetilde R_{\eps_n}u(2j) = \sum_{\l \ge 1} a_\l(\eps_n)_{2j+\l} .$$
Thus 
\begin{eqnarray*}
Q_{\eps_n}^mu^2(0)-(Q_{\eps_n}^mu(0))^2 = 2m+\kO(n^{1/2}).
\end{eqnarray*}
Then \eqref{VkQkbis} shows that for $a\le x\le 0\wedge \tau_n^R$, 
\begin{eqnarray*}
A^{(n)}_{a,v}(x) &= & \frac 1{n^2} \sum_{k=1}^{[2nx]-[2na]} \left(Q_{\eps_n}^{V_{\eps_n}(k-1)}u^2(0) - (Q_{\eps_n}^{V_{\eps_n}(k-1)}u(0))^2\right)\\
 &= & \frac 2 {n} \sum_{k=1}^{[2nx]-[2na]} \Lambda_{a,v}^{(n)}\left(a+\frac {k-1}{2n}\right) + \kO\left(\frac 1 {\sqrt{n}}\right)\\
 &=& 4\int_a^x \Lambda^{(n)}_{a,v}(y) \, dy + \kO\left(\frac 1 {\sqrt{n}}\right).
\end{eqnarray*}
This finishes the proof of Proposition \ref{propB}. \hfill $\square$


\subsection{Proof of the convergence on $[0,+\infty)$}
\label{secainfini}
The proof of the convergence of $\Lambda^{(n)}_{a,v}$ on $[0,+\infty)$ is essentially the same as the proof on $[a,0]$. Namely we can define $\widetilde M^{(n)}_{a,v}$, $\widetilde B_{a,v}^{(n)}$ and 
$\widetilde A_{a,v}^{(n)}$, respectively as in \eqref{Bn}, \eqref{An} and \eqref{A} with $\widetilde V$ everywhere instead of $V$. 
Let also 
\begin{eqnarray}
\label{wn+-}
\left\{ 
\begin{array}{l}
w^{(n,-)}_{a,v}:=\frac 1 {2n} \sup\{k\le 0\ :\  S_{\eps_n,[2na],[nv]}(k)=0\} \\
w^{(n,+)}_{a,v}:=\frac 1 {2n}\inf\{k\ge a\ :\  S_{\eps_n,[2na],[nv]}(k)=0\}.
\end{array}
\right.
\end{eqnarray}
Then 
$$\Lambda^{(n)}_{a,v}(x)= \Lambda^{(n)}_{a,v}(0)+ \widetilde M^{(n)}_{a,v}(x)+ \widetilde B_{a,v}^{(n)}(x)\quad \text{for all }x\in [0, w^{(n,+)}_{a,v}).$$
Moreover \eqref{VW2} shows that 
$$
\E[\widetilde V_{\eps}(k)\mid \widetilde V_{\eps}(k-1)]-\widetilde V_{\eps}(k-1) = Q_{\epsilon}^{\widetilde V_{\eps}(k-1)}u(0)-Q_0^{\widetilde V_{\eps}(k-1)}u(0),
$$
and
$$
\E[\widetilde V_{\eps}(k)^2 \mid \widetilde V_{\varepsilon}(k-1)]- \E[\widetilde V_{\varepsilon}(k)\mid \widetilde V_{\varepsilon}(k-1)]^2= Q_{\eps}^{\widetilde V_{\eps}(k-1)}u^2(0) - (Q_{\eps}^{\widetilde V_{\eps }(k-1)}u(0))^2,
$$
for all $k\ge 1$. Then by following the proofs given in the previous subsections we get the analogues of Proposition \ref{propA} and \ref{propB}: 

\begin{prop}
\label{prop2}
Let $R>0$ and $T>0$ be given. Then for $0\le x\le T\wedge \tau_n^R\wedge w^{(n,+)}_{a,v}$, 
$$\widetilde B^{(n)}_{a,v}(x)= 2\int_0^x h(\Lambda^{(n)}_{a,v}(y))\, dy + \mathcal{O}\left(\frac{1}{\sqrt{n}}\right),$$
where the $\kO(n^{-1/2})$ is deterministic and only depends on $a$, $T$ and $R$. 
\end{prop}

\begin{prop} 
\label{prop2bis}
Let $R>0$ and $T>0$ be given. Then for $0\le x\le T\wedge \tau_n^R\wedge w^{(n,+)}_{a,v}$, 
$$\widetilde A^{(n)}_{a,v}(x) = 4\int_0^x\Lambda^{(n)}_{a,v}(y)\, dy + \mathcal{O}\left(\frac{1}{\sqrt{n}}\right),$$
where the $\kO(n^{-1/2})$ is deterministic and only depends on $a$, $T$ and $R$.
\end{prop} 


\noindent So according again to the criterion of Ethier and Kurtz (Theorem 4.1 p.354 in \cite{EK}), we deduce the convergence in law of $\Lambda^{(n)}_{a,v}$ on $[0,+\infty)$.

\noindent Actually one can deduce the convergence on $[a,+\infty)$ as well. For this we just need to observe that the criterion of Ethier and Kurtz 
applies in the same way for non homogeneous operators. For reader's convenience let us recall the main steps of its proof. First Propositions 
\ref{propA}, \ref{propB}, \ref{prop2} and \ref{prop2bis} imply the tightness of the sequence $(\Lambda^{(n)}_{a,v}, n\ge 1)$ on $[a,+\infty)$. 
Next It\^o Formula shows that any 
limit of a subsequence is a solution of the non-homogeneous martingale problem (see the definition in \cite{EK} p.221) associated with the operator    
\begin{eqnarray*}
G_xf(\lambda) :=  \left\{ \begin{array}{ll} 2\lambda f''(\lambda) + 2(1+h(\lambda))f'(\lambda) & \text{if } x\in [a,0] \\
                                       2\lambda f''(\lambda) + 2h(\lambda)f'(\lambda)  & \text{if } x\in [0,+\infty).
\end{array} 
\right. 
\end{eqnarray*} 
Then Theorem 2.3 p.372 in \cite{EK} (with their notation replace $t$ by $x$, $x$ by $\lambda$ and take $r_0=0$ and $r_1=+\infty$) shows that this martingale problem is well posed (in particular it has a unique solution). This proves the desired convergence on $[a,+\infty)$. 
Since the proof of the convergence on $(-\infty,a]$ is the same as on $[0,+\infty)$, this concludes the proof of Theorem \ref{theo}.  \hfill $\square$

\section{Extension to the non homogeneous setting}
\label{subsecnonh}
We give here an extension of Theorem \ref{theo} when $|I|=1$ and $\varphi$ is allowed to be space dependent. Apart from its own interest, we will use this extension to prove Theorem \ref{theo} when $|I|\ge 2$, see the next section.   

\vspace{0.2cm} 
\noindent We first define non homogeneous cookies random walks. If 
$$\eps = (\eps_{i,x},i\ge 1,x\in \Z),$$
is given, we set
$$p_{\eps,i,x}:=\frac 1 2 (1+\eps_{i,x}),$$
for all $i$ and $x$. 
Then $X_\eps$ is defined by 
$$\pp[X_\eps(n+1)-X_\eps(n) = 1 \mid \kF_{\eps,n}]= 1 -  \pp[X_\eps(n+1)-X_\eps(n) = -1 \mid \kF_{\eps,n}] = p_{\eps,i,x},$$
if $\#\{j\le n\ :\ X_\eps(j)=X_\eps(n) \} = i$ and $X_\eps(n)=x$.
Similarly non homogeneous excited Brownian motions are defined by 
$$dY_t = dB_t + \varphi(Y_t, L_t^{Y_t})\, dt,$$
for some bounded and measurable $\varphi$. Such generalized version of excited BM was already studied in \cite{NRW2} and \cite{RS}. In particular
Ray-Knight results were obtained in this context and a sufficient condition for recurrence is given in \cite{RS} (see below). Now let $\varphi$ be some fixed bounded c\`adl\`ag function. 
Assume that for each $n\ge 1$, a function $\varphi_n:\Z\times [0,\infty)\to \R$, c\`adl\`ag in the second variable, is given. Consider 
$$\Delta_n(x) : = \sup_\l |\varphi_n([2nx],\l)-\varphi(x,\l)|,$$
and assume that   
\begin{eqnarray}
\label{cvphin}
\Delta_n  \to  0 \quad \textrm{in }\D(\R)\quad \textrm{when }n\to \infty.
\end{eqnarray}
Assume further that $\sup_{k,\l} |\varphi_n(k,\l)| < 2n$ for $n$ 
large enough and define  
$\eps_n=(\eps_{i,x}(n),i\ge 1,x\in \Z)$ by 
\begin{eqnarray}
\label{epsninhomogene}
\eps_{i,x}(n)= \frac 1 {2n} \varphi_n\left(x, \frac{i}{2n}\right),
\end{eqnarray} 
for all $i\ge 1$ and $x\in \Z$. Say that $\varphi$ is uniformly Lipschitz in the second variable if 
\begin{eqnarray}
\label{uniformeLip}
\sup_{x\in \R}\ \sup_{\l\neq \l'} \frac {|\varphi(x,\l)-\varphi(x,\l')|}{|\l-\l'|} <+\infty.
\end{eqnarray}
Finally define $\Lambda^{(n)}_{a,v}$ and $\Lambda_{a,v}$ as in the homogeneous setting (see the introduction).
We can state now the following natural extension of Theorem \ref{theo}: 

\begin{theo} 
\label{theoext} 
Let $\varphi$ be some bounded c\`adl\`ag function satisfying \eqref{uniformeLip}. Assume that for $n$ large enough, $X_{\eps_n}$ is recurrent and that $Y$ is recurrent.  Assume further that \eqref{cvphin} holds. Then for any $a\in \R$ and $v\ge 0$, 
$$(\Lambda^{(n)}_{a,v}(x),x\in \R)\mathop{\Longrightarrow}_{n\rightarrow\infty}^{\mathcal{L}} \left(\Lambda_{a,v}(x),x\in \R\right).$$ 
\end{theo}

\noindent The proof of this result is exactly the same as the proof of Theorem \ref{theo}. Note that as at the end of the previous subsection, we need to use here a non homogeneous version of Ethier--Kurtz's result (Theorem 4.1 p.354 in \cite{EK}). This time we just have to verify that the martingale problem associated with the operator 
\begin{eqnarray*}
G_xf(\lambda) :=  \left\{ \begin{array}{ll} 2\lambda f''(\lambda) + 2(1+h(x,\lambda))f'(\lambda) & \text{if } x\in [a,0] \\
                                       2\lambda f''(\lambda) + 2h(x,\lambda)f'(\lambda)  & \text{if } x\in [0,+\infty),
\end{array} 
\right. 
\end{eqnarray*} 
is well posed, where $h(x,\lambda)=\int_0^\lambda \varphi(x,\mu)\, d\mu$, for any $x$ and $\lambda$. But again this follows from Theorem 2.3 p.372 in \cite{EK}.

\vspace{0.2cm}
\noindent In particular the above theorem applies to the following situation, which we will use in the proof of 
Theorem \ref{cvlawfd}. 
Assume that $\varphi:\R\times [0,\infty)\to \R$ satisfies the hypotheses of Theorem \ref{theoext} and that 
a sequence $(\varphi_n,n\ge 1)$ converges to $\varphi$ as in \eqref{cvphin}. 
Assume in addition that for each $n\ge 1$, a function $(\lambda(n,x),x\in \Z)$ is given. 
Set $\lambda_n:=\lambda(n,[2n\cdot])$ and assume further that there exists $\lambda \in \D(\R)$ such that 
\begin{eqnarray}
\label{cvphin'}
\lambda_n \to \lambda \quad \textrm{in }\D(\R)\quad \textrm{when }n\to \infty.
\end{eqnarray}
Set
$$\varphi_\lambda(x,\l):=\varphi(x,\lambda(x)+\l) \quad \textrm{for all } x\in \R \textrm{ and }\l \ge 0,$$
and
$$\varphi'_n(x,\l):= \varphi_n(x,\lambda(n,x)+\l) \quad \textrm{for all } x\in \Z \textrm{ and }\l \ge 0.$$
Note that if \eqref{cvphin'} holds, then 
$\varphi'_n([2n\cdot],\cdot)$ converges to $\varphi_\lambda$ as in \eqref{cvphin}.  
Let now $\eps_{n,\lambda_n}=(\eps_{i,x}(n,\lambda_n),i\ge 1, x\in \Z)$ be defined by
$$
\eps_{i,x}(n,\lambda_n):= \frac 1 {2n} \varphi'_n\left(x,\frac{i}{2n}\right).
$$  
Let $\Lambda^{(n,\lambda_n)}$ and $\Lambda^{(\lambda)}$ be the processes associated to $\eps_{n,\lambda_n}$ and $\varphi_\lambda$ as in the introduction. 
The following is an immediate application of Theorem \ref{theoext}: 
\begin{cor} 
\label{corext}
Assume that the hypotheses of Theorem \ref{theoext} and \eqref{cvphin'} hold true. 
Then with the above notation, for any $a$ and $v$, 
$$\left(\Lambda^{(n,\lambda_n)}_{a,v}(x),x\in \R\right) \mathop{\Longrightarrow}_{n\rightarrow\infty}^{\mathcal{L}}\left(\Lambda^{(\lambda)}_{a,v}(x),x\in \R\right).$$
\end{cor}  

\begin{rem} \label{reminhomogene} \emph{Actaully the result of this corollary holds as well in the slightly more general setting where $v$ is not fixed. More precisely, if $v_n$ converges to $v$ when $n\to \infty$, then $\Lambda^{(n,\lambda_n)}_{a,v_n}$ also converges in law toward $\Lambda^{(\lambda)}_{a,v}$. The proof is exactly the same, since this setting is covered by Ethier-Kurtz's result. }
\end{rem} 

\noindent To finish this section, we recall some sufficient condition for recurrence of $X_\eps$ and $Y$ proved respectively in \cite[Corollary 7]{Z} and \cite[Corollary 5.6]{RS} in the non homogeneous case. We notice that it applies only when for all $i$ and $x$, $\eps_{i,x}$, respectively $\varphi$, is nonnegative. We only state the result in the continuous setting, the result for $X_\eps$ being analogous. So if for $x\in \R$,  
$$\delta^x(\varphi):=\int_0^\infty \varphi(x,\l)\ d\l,$$ 
then $Y$ is recurrent as soon as 
$$\liminf_{z\to +\infty} \frac{1}{z}\int_0^z \delta^x(\varphi)\ dx<1.$$


\section{Proof of Theorem \ref{theo} in the general case}
\label{Ige2}
We actually prove the result in the non homogeneous setting. 

\begin{theo} \label{cvlawfd} Under the hypotheses of Theorem \ref{theoext}, for any finite set $I$, any $a_i\in \R$ and $v_i\ge 0$, $i\in I$, 
$$\left(\Lambda^{(n)}_{a_i,v_i}(x),x\in \R\right)_{i\in I}
\quad \mathop{\Longrightarrow}_{n\rightarrow\infty}^{\mathcal{L}} \quad \left(\Lambda_{a_i,v_i}(x),x\in \R\right)_{i\in I}.$$
\end{theo}
\begin{proof} When $|I|=1$, the result is given by Theorem \ref{theoext}.
The general case can then be proved by induction on the cardinality of $I$. 
To simplify the notation we only make the proof of the induction step when the cardinality of $I$ equals $2$, but it would work similarly in general.  
So let $a$, $a'$, $v$ and $v'$ be given. All we have to prove is that for any continuous and bounded functions $H$ and $\widetilde H$, 
\begin{eqnarray}
\label{convHtildeH}
\E\left[H\left(\Lambda^{(n)}_{a',v'}\right)\widetilde H\left(\Lambda^{(n)}_{a,v}\right)\right] \to 
\E\left[H\left(\Lambda_{a',v'}\right)\widetilde H\left(\Lambda_{a,v}\right)\right],
\end{eqnarray}
when $n\to \infty$. Consider the events 
$$\kA_n:=\left\{\Lambda^{(n)}_{a,v}(a')<v'\right\},$$
for $n\ge 1$, and 
$$\kA:=\left\{\Lambda_{a,v}(a')<v'\right\}.$$
\noindent Observe that conditionally to $\Lambda^{(n)}_{a,v}$ and on the set $\kA_n$ we have the equality in law: 
\begin{eqnarray}
\label{egaliteloiaa}
\Lambda^{(n)}_{a',v'}-\Lambda^{(n)}_{a,v} = \Lambda^{(n,\Lambda^{(n)}_{a,v}(a+\, \cdot ) )}_{a'-a,v'-\Lambda^{(n)}_{a,v}(a')},
\end{eqnarray}
with the notation of Corollary \ref{corext}. This identity is straightforward. Maybe less immediate is the analogous equality in the continuous setting, so we state it as a lemma: 
\begin{lem} 
Let $a$, $a'$, $v$ and $v'$ be given. Conditionally to $\Lambda_{a,v}$ and on $\kA$, we have the equality in law: 
\begin{eqnarray}
\label{la'v'lav}
\Lambda_{a',v'}-\Lambda_{a,v} = \Lambda^{(\Lambda_{a,v}(a+\, \cdot ) )}_{a'-a,v'-\Lambda_{a,v}(a')}.
\end{eqnarray}
\end{lem} 
\begin{proof} 
One just has to observe (see also (2) in \cite{RS}) that conditionally to $\Lambda_{a,v}$ and on $\kA$, the law of $(Y_{t+\tau_a(v)},t\ge 0)$ is 
equal to the law of an excited BM starting from $a$ and associated to the non homogeneous function $\widetilde \varphi$ defined by 
$$\widetilde \varphi(x,\l) = \varphi(x,\Lambda_{a,v}(x)+\l).$$
The lemma follows.  
\end{proof} 
\noindent It follows from \eqref{egaliteloiaa} that for any continuous and bounded $H$, 
$$\E\left[H\left(\Lambda^{(n)}_{a,v}+ (\Lambda^{(n)}_{a',v'}-\Lambda^{(n)}_{a,v})\right)\ \Big|\ \Lambda^{(n)}_{a,v}\right]1_{\kA_n} = \HH_n\left(\Lambda^{(n)}_{a,v}\right)1_{\kA_n},$$
where 
$$\HH_n\left(\lambda\right) := \E\left[H\left(\lambda+ \Lambda_{a'-a,v'-\lambda(a')}^{(n,\lambda(a+\cdot))}\right) \right],$$
for any $\lambda$ in the Skorokhod space $\D(\R)$ such that $\lambda(a')\le v'$. Define similarly $\HH$ by 
$$\HH(\lambda) :=  \E\left[H\left(\lambda + \Lambda_{a'-a,v'-\lambda(a')}^{(\lambda(a+\cdot))}\right)\right],$$
for any $\lambda \in \D(\R)$ such that $\lambda(a')\le v'$. 
Now Corollary \ref{corext} (see also the remark following it) shows that for any sequence of functions $\lambda_n$, satisfying $\lambda_n(a')\le v'$, and  
converging to some $\lambda$ in $\D(\R)$, $\HH_n(\lambda_n)$ converges toward $\HH(\lambda)$. 
Moreover, by using the Skorokhod's representation theorem (see Theorem 6.7 in \cite{Bil}), we can assume that $\Lambda^{(n)}_{a,v}$ converges almost surely toward $\Lambda_{a,v}$. 
We claim that $1_{\kA_n}$ also converges a.s. to $1_{\kA}$. To see this it suffices to prove that $\pp[\Lambda_{a,v}(a')=v']=0$. But the set $\{\Lambda_{a,v}(a')=v'\}$ is included in the set $\{e_{a'}(v')\neq 0\}$, where $e_{a'}(v')$ denotes the excursion of 
$Y$ out of level $a'$ starting from $\tau_{a'}(v'-)$, and this last event has probability $0$ (this is well known to be the case for the Brownian motion, and can be deduced for $Y$ by an absolute continuity argument, see also \cite{RS}).

So if $H$ and $\widetilde H$ are two continuous and bounded functions, we deduce from the dominated convergence theorem that  
\begin{eqnarray}
\label{HHtilden}
\E\left[H\left(\Lambda^{(n)}_{a',v'}\right)\widetilde H\left(\Lambda^{(n)}_{a,v}\right) 1_{\kA_n}\right] 
\to  
\E\left[H\left(\Lambda_{a',v'}\right)\widetilde H\left(\Lambda_{a,v}\right) 1_{\kA}\right],
\end{eqnarray}
when $n\to \infty$. By using the same argument we see that the convergence in \eqref{HHtilden} also holds if we replace $\kA_n$ and $\kA$ respectively by $\kA_n^c$ and $\kA^c$. Then \eqref{convHtildeH} follows and this concludes the proof of Theorem \ref{cvlawfd}.
\end{proof}

\section{Proof of Corollary \ref{cor1}}
\label{seccor} 

Note that for any $a\in \R$ and $v\ge 0$, 
$$\tau_{\eps_n,[2na]}([nv])=[2na]+ 2\sum_{k\in \Z} S_{\eps_n,[2na],[nv]}(k).$$
Thus 
\begin{eqnarray*}
\label{conv1}
\frac{\tau_{\eps_n,[2na]}([nv])}{4n^2} = \int_\R \Lambda^{(n)}_{a,v}(y)\, dy+o(1).
\end{eqnarray*}
On the other hand, the occupation times formula (\cite{RY} p.224) gives
\begin{eqnarray*}
\label{conv2}
\tau_{a}(v)=\int_\R \Lambda_{a,v}(y)\, dy.
\end{eqnarray*} 
Now Theorem \ref{theo} shows that for any $a_i$, $v_i$, $i\in I$, and any fixed $A>0$, the following convergence in law holds: 
\begin{eqnarray*}
\label{convA}
\left(\int_{-A}^A \Lambda^{(n)}_{a_i,v_i}(y)\, dy\right)_{i\in I} \quad  \mathop{\Longrightarrow}_{n\rightarrow\infty}
^{\mathcal{L}} \quad \left(\int_{-A}^A \Lambda_{a_i,v_i}(y)\, dy\right)_{i\in I}.
\end{eqnarray*} 
So Corollary \ref{cor1} follows from the following lemma (recall that $w^{(n,\pm)}_{a,v}$ is defined in \eqref{wn+-}): 
\begin{lem} 
\label{lemwn}
Let $\epsilon>0$, $a\in \R$ and $v\ge 0$ be given. Then there exists $A>0$, such that 
$$\pp\left[|w^{(n,\pm)}_{a,v}|\ge A\right]\le \epsilon,$$
for all $n$ large enough. 
\end{lem} 
\begin{proof}
We prove the result for $w^{(n,+)}_{a,v}$. The proof for $w^{(n,-)}_{a,v}$ is the same. 
First observe that $w_{a,v}^+$ is nonnegative and a.s. finite: it is equal to $\sup \{ Y_t\ :\ t\le \tau_{a}(v)\}$ and 
$\tau_{a}(v)$ is a.s. finite since $Y$ is recurrent. So for any $\epsilon>0$, there exists $A>a$ such that 
\begin{eqnarray*}
\label{wav2}
\pp[w_{a,v}^+\ge A]\le \epsilon.
\end{eqnarray*}
Moreover by using Theorem \ref{theo} and Skorokhod's representation Theorem, it is possible to define $\Lambda^{(n)}_{a,v}$ and $\Lambda_{a,v}$ on the same probability space, such that for any $\eta>0$, 
\begin{eqnarray*}
\label{wav3}
\pp\left[\sup_{0\le x\le A} |\Lambda^{(n)}_{a,v}(x)-\Lambda_{a,v}(x)|\ge \eta\right]\le \epsilon,
\end{eqnarray*}
for $n$ large enough. Thus 
\begin{eqnarray}
\label{teta}
\pp[T_n(\eta)\ge A] \le 2\epsilon,
\end{eqnarray}
where
$$T_n(\eta)=\inf\{x>0 \ :\ \Lambda^{(n)}_{a,v}(x) \le \eta\}.$$ 
Recall now that on $[0,+\infty)$, $\Lambda^{(n)}_{a,v}(\cdot)$ is equal in law to $\widetilde V_{\eps_n}([2n \cdot])/n$ (see the beginning of Section \ref{Sec2.1}). But since $|\eps_n|_\infty =\kO(1/n)$, $(\widetilde V_{\eps_n}(k),k\ge 0)$ is stochastically dominated by a Galton-Watson process $(W_n(k),k\ge 0)$ with offspring distribution 
a geometrical law with parameter $1-p_n=1/2-c/n$, for some constant $c>0$ (with the convention that if $G$ is a random variable with such geometrical law, then $\pp(G=k) = p_n(1-p_n)^k$ for all $k\ge 0$, in particular $\E(G)= (1-p_n)/p_n<1$).  
Moreover, when $W_n(0)=1$, the probability for $W_n$ to extinct before time $[nA]$ can be computed explicitly. 
If $f^{(n)}_k(\cdot)$ is the generating function of $W_n(k)$, then this probability is equal to $f^{(n)}_{[nA]}(0)$. An expression for $f^{(n)}_k(0)$ is given for instance in \cite{AN} p.6-7:
$$f^{(n)}_k(0)= 1 - m_n^k\frac{1-s_n}{m_n^k-s_n} \quad \textrm{for all }k\ge 1,$$
where 
$$m_n = \frac{p_n}{1-p_n}= 1+ \frac{4c}{n}+ \kO\left(\frac 1 {n^2}\right),$$
and 
$$s_n=\{1-m_n(1-p_n)\}/p_n=1- \frac{4c}{n}+ \kO\left(\frac 1 {n^2}\right).$$
It follows that $f^{(n)}_{[nA]}(0) = 1- c'/n + \kO(1/n^2)$, with $c'=4c/(1-e^{-4cA})>0$. Now the law of $W_n$ starting from $[\eta n]$ is equal to the law of the sum of 
$[\eta n]$ independent copies of $W_n$ starting from $1$. Thus if $W_n(0)=[\eta n]$, the probability for $W_n$ to extinct before time $[nA]$ is $f^{(n)}_{[nA]}(0)^{[\eta n]}$.
If $\eta$ is small enough and $n$ large enough, this probability is larger than $(1-\epsilon)$. 
By using now that $\widetilde V_{\eps_n}$ is stochastically dominated by $W_n$, \eqref{teta} and the strong Markov property, we get 
\begin{eqnarray*}
\pp\left[w^{(n,+)}_{a,v}\ge 2A\right] & \le &  \pp\left[w^{(n,+)}_{a,v}\ge 2A \textrm{ and } T_n(\eta) \le A\right] +\pp\left[ T_n(\eta) \ge A\right]\\ 
  &\le & \pp\left[ W_n([nA]) >0 \mid W_n(0)= [\eta n] \right] + 2 \epsilon \le 3\epsilon. 
\end{eqnarray*}
This concludes the proof of the lemma.  
\end{proof}


\section{Proof of Corollary \ref{cor2}}
\label{seccor2}
To simplify notation we only prove the result when $|I|=1$ but the general case works the same. 
First note that for any $\lambda$, the law of $Y_{\gamma_\lambda}$ has for density the function $a\mapsto \lambda\E[L^a_{\gamma_\lambda}]$. Indeed, for any bounded and measurable function $\phi$,
\begin{eqnarray*}
\E[\phi(Y(\gamma_\lambda))]
&=& \E\left[\int_0^\infty \phi(Y(s))\lambda e^{-\lambda s}\, ds\right]\\
&=& \E\left[\int_{0<s<t} \lambda^2\phi(Y(s)) e^{-\lambda t}\, ds\, dt\right]\\
&=& \E\left[\int_\R\int_0^\infty \lambda^2 \phi(a)L^a_t e^{-\lambda t}\, dt\, da\right]\\
&=& \int_\R\E[ \phi(a)\lambda L^a_{\gamma_\lambda}]\, da,
\end{eqnarray*}
where in the third equality we have used the occupation times formula (see Corollary (1.6) p.224 in \cite{RY}).

We now follow the argument given by T\'oth in \cite{T}. 
First observe that if  
$$\widetilde \tau_{\eps,a}(v):= \inf\left\{ j \ :\ \#\{i\le j\ :\ X_\eps(i)=a\text{ and } X_\eps(i+1)=a+1\}=v+1\right\},$$
then exactly as we proved Corollary \ref{cor1}, we can show that $\widetilde \tau_{\eps_n,[2na]}([nv])/(4n^2)$ converges in law toward $\tau_a(v)$ for any $a\in \R$ and $v\ge 0$.
Next observe that for any $a\in \Z$ and $k\in \N$,
$$\pp[X_{\eps_n}(k) = a] = \sum_{v\in \N}\left\{ \pp [\tau_{\eps_n,a}(v) = k]+ \pp [\widetilde \tau_{\eps_n,a}(v) = k]\right\}.$$
Thus for any $a\in \R$,
\begin{eqnarray*}
2n\pp\left(X_{\eps_n}(\theta_{\lambda/(4n^2)})=[2na]\right) & = & 2n(1-e^{-\lambda/(4n^2)}) \sum_{k\ge 0} e^{-k\lambda/(4n^2)}\\
      & & \times \sum_{v\in \N}\left\{ \pp [\tau_{\eps_n,[2na]}(v) = k]+ \pp [\widetilde \tau_{\eps_n,[2na]}(v) = k]\right\}\\
      &\sim & \frac{\lambda}{2n} \sum_{v\in \N} \left\{\E\left[e^{-\lambda \frac{\tau_{\eps_n,[2na]}(v)}{4n^2}}\right]+\E\left[e^{-\lambda \frac{\widetilde \tau_{\eps_n,[2na]}(v)}{4n^2}}\right]\right\},
\end{eqnarray*}
since $2n(1-e^{-\lambda/(4n^2)}) \sim \lambda/(2n)$. Note now that 
$$\frac 1 n \sum_{v\in \N} \E\left[e^{-\lambda \frac{\tau_{\eps_n,[2na]}(v)}{4n^2}}\right] = \int_0^\infty \E\left[e^{-\lambda \frac{\tau_{\eps_n,[2na]}([nv])}{4n^2}}\right]\, dv,$$
and that for any $v\in \R^+$, Corollary \ref{cor1} implies  
$$\E\left[e^{-\lambda \frac{\tau_{\eps_n,[2na]}([nv])}{4n^2}}\right] \to \E\left[e^{-\lambda \tau_{a}(v)}\right],$$
when $n\to \infty$. The same remark applies with $\widetilde{\tau}$ instead of $\tau$. Thus by application of Fatou's lemma, for every $a\in \R$,  
\begin{eqnarray}
\label{11}
\liminf_{n\to \infty} \,  (2n)\pp\left(X_{\eps_n}(\theta_{\lambda/(4n^2)})=[2na]\right) \ge \lambda\int_0^\infty \E\left[e^{-\lambda \tau_{a}(v)}\right]\, dv.
\end{eqnarray}
But notice that for every $a\in \R$ and $v\ge 0$, 
\begin{eqnarray*}
\E\left[e^{-\lambda \tau_{a}(v)}\right]&=& \lambda\int_0^\infty e^{-\lambda s}\pp[\tau_{a}(v)\le s]\ ds \\
         & = & \lambda\int_0^\infty e^{-\lambda s} \pp[L_s^{a} \ge v]\ ds\\
         &=& \pp[L^{a}_{\gamma_\lambda} \ge v].
\end{eqnarray*}
Therefore
\begin{eqnarray}
\label{12}
\lambda\int_\R \int_0^\infty \E\left[e^{-\lambda \tau_{a}(v)}\right]\ dv \ da = \lambda\int_\R  \E[L^{a}_{\gamma_\lambda}]\ da = \lambda\E[\gamma_\lambda] = 1.
\end{eqnarray}
On the other hand for any $n$,
\begin{eqnarray}
\label{13}
\int_\R  (2n)\pp\left(X_{\eps_n}(\theta_{\lambda/(4n^2)})=[2na]\right)\, da =1.
\end{eqnarray}
It follows now from \eqref{11} \eqref{12} and \eqref{13} that for almost every $a\in \R$, 
$$\lim_{n\to \infty} (2n)\pp\left(X_{\eps_n}(\theta_{\lambda/(4n^2)})=[2na]\right) =\lambda\E[L^{a}_{\gamma_\lambda}].$$
The corollary is then a consequence of Sheff\'e's lemma.  \hfill $\square$

\section{Proof of Theorem \ref{cor3}}
\label{seccor3} 
We actually prove the following extension of Theorem \ref{cor3} in the non homogeneous setting: 
\begin{theo} 
\label{theofinal} Let $\varphi$ be some bounded c\`adl\`ag function satisfying \eqref{uniformeLip}. Let also $(\varphi_n)_{n\ge 1}$ be a sequence of bounded c\`adl\`ag functions converging to $\varphi$ as in \eqref{cvphin}. Let $\eps_n$ be defined by \eqref{epsninhomogene} and for $t\ge 0$, set $X^{(n)}(t):=X_{\eps_n}([4n^2t])/(2n)$. Then 
$$(X^{(n)}(t),t\ge 0)\quad  \mathop{\Longrightarrow}_{n\rightarrow\infty}^{\mathcal{L}}\quad (Y(t),t\ge 0).$$ 
\end{theo}
\begin{proof} We first assume that $Y$ is recurrent and that $X_{\eps_n}$ is recurrent as well at least for $n$ large enough. 
We will see below how one can then remove this hypothesis by using a truncation argument. In this proof we will use Corollaries \ref{cor1} and \ref{cor2}, and their extension to the non homogeneous setting (these being straightforwards).

\vspace{0.2cm}
\noindent \textit{Tightness:} We first need to show that the sequence $(X^{(n)}(\cdot),n\ge 1)$ is tight. All we have to prove (see e.g. Lemma (1.7) p.516 in \cite{RY}) is that for each $T>0$, $\alpha>0$ and $\eta>0$, there are $n_0$ and $\kappa>0$, such that for $n\ge n_0$, 
\begin{eqnarray}
\label{tension}
\pp\left[\sup_{t\le s\le t+ \kappa} |X^{(n)}(s)-X^{(n)}(t)|\ge \eta\right] \le \alpha\kappa \quad \textrm{for all }t\le T.
\end{eqnarray}
We first prove the above inequality for $t=0$. For this it suffices to find $\kappa>0$ such that 
\begin{eqnarray*}
\pp\left[\tau_{\eps_n,[2n\eta]}(0)\wedge \widetilde \tau_{\eps_n,[2n\eta]}(0) \le 4 n^2 \kappa \right] \le \alpha\kappa,
\end{eqnarray*}
for $n$ large enough ($\eta>0$ and $\alpha$ being arbitrary and fixed), since the analogous result for $\eta<0$ is similar (use the same proof with the process $-X^{(n)}$ instead of $X^{(n)}$). In fact it suffices to prove that 
\begin{eqnarray}
\label{tensionbis}
\pp\left[\tau_{\eps_n,[2n\eta]}(0) \le 4 n^2 \kappa \right] \le \alpha\kappa/2,
\end{eqnarray}
since the result with $\widetilde \tau_{\eps_n,[2n\eta]}(0)$ in place of $\tau_{\eps_n,[2n\eta]}(0)$ is similar. 
A basic coupling shows that if $\sup \varphi \le C$, for some constant $C>0$, then the probability on the left hand side of \eqref{tensionbis} is 
bounded by the analogous probability we would get by taking $\varphi$ constant equal to $C$. 
But it is well known (this follows also from Corollary \ref{cor1}) that as $n$ tends to $\infty$, the left hand side in \eqref{tensionbis} converges toward $\pp[\tau_\eta(0) \le \kappa]$ and that this last term is a $o(\kappa)$ for Brownian motion with constant drift (see for instance Proposition (3.7) p.105 in \cite{RY}). This proves \eqref{tensionbis}. 
To obtain \eqref{tension} it suffices to observe that after time $t$, $X^{(n)}$ is equal in law to a renormalized non-homogeneous cookie random walk starting from $X^{(n)}(t)$ and 
evolving in a shifted cookie environment (see also \eqref{egaliteloiaa}). So we can apply the same proof and we obtain the same result with the same constants everywhere. This finishes to prove the tightness of $(X^{(n)}(\cdot),n\ge 0)$.  
It remains to prove the convergence of the 
finite-dimensional distributions:

\vspace{0.2cm} \noindent \textit{Convergence of finite-dimensional distributions:} To simplify notation we prove the result for one-dimensional distributions, but it works the same in general. So let $(W_t,t\ge 0)$ be some limit in law of $(X^{(n_k)}(t),t\ge 0)$, for a subsequence $(n_k,k\ge 0)$. 
Then for any bounded and measurable function $\phi$,  
\begin{eqnarray*}
\E\left[\phi\left( \frac{X_{\eps_{n_k}}(\theta_{\lambda/(4{n_k}^2)})}{2n_k}\right)\right] &\sim_{k\to \infty} & \lambda\int_0^\infty e^{-\lambda t}\E\left[\phi(X^{(n_k)}(t))\right]\, dt\\
&\to_{k\to \infty} & \lambda \int_0^\infty e^{-\lambda t}\E[\phi(W_t)] \, dt\\
&=& \E[\phi(W_{\gamma_\lambda})].
\end{eqnarray*}  
On the other hand Corollary \ref{cor2} shows that the term on the left hand side converges toward $\E[\phi(Y(\gamma_\lambda))]$. Since this holds for any $\lambda$ and any $\phi$, we deduce that $W_t$ and $Y(t)$ have the same law for every $t\ge 0$ (see \cite[Theorem 1a p.432]{F}). This proves the convergence of  one-dimensional distributions. 

\vspace{0.2cm}
\noindent This finishes the proof under the additional hypothesis of recurrence and it just remains to explain how one can remove this hypothesis. 
For this we use a truncation argument. 
For any $R>0$ and $n\ge 1$, let $\varphi_R(x,\l)$, resp. $\varphi_{n,R}(x,\l)$, be the functions equal to $\varphi(x,\l)$, resp. $\varphi_n(x,\l)$, when $x\in[-R,R)$, resp. when $|x|\le 2nR$, and equal to zero otherwise. It is immediate that 
$\varphi_{n,R}$ still converges to $\varphi_R$ as in \eqref{cvphin}. 
Denote now by $X^{(n)}_R$ and $Y_R$ the processes associated respectively to $\varphi_{n,R}$ and $\varphi_R$. 
Since $\varphi$ and the $\varphi_n$'s are bounded, and since they are equal to zero outside of $[-R,R)$, these processes are recurrent. So we just have proved that $X^{(n)}_R$ converges in law toward $Y_R$. 
Note now that up to the time 
$(\widetilde \tau_{\eps_n,[2nR]}(0)\wedge \tau_{\eps_n,[-2nR]}(0))/(4n^2)$, (with the notation from section \ref{seccor2}), 
$X^{(n)}_R$ and $X^{(n)}$ are equal in law. Similarly up to the time $\tau_{R}(0)\wedge \tau_{-R}(0)$, $Y_R$ and $Y$ are equal in law. But Corollary \ref{cor1} shows that 
$\widetilde \tau_{\eps_n,[2nR]}(0)/(4n^2)$ and $\tau_{\eps_n,[-2nR]}(0)/(4n^2)$, converge in law respectively toward $\tau_{R}(0)$ and $\tau_{-R}(0)$. By using also the monotonicity in $R$ of these random times, we deduce that for any $T>0$, 
$$\pp\left(\widetilde \tau_{\eps_n,[2nR]}(0)\wedge \tau_{\eps_n,[-2nR]}(0) \le 4n^2T\right) \to 0,$$
when $R\to \infty$, uniformly in $n$. It follows immediately that for any $T>0$, $X^{(n)}$ converges in law to $Y$ on the time interval $[0,T]$. Since this is true for any $T>0$, this concludes the proof of Theorem \ref{theofinal}.  
\end{proof}

\begin{rem} 
\emph{The notation $\Lambda_{a,v}(x)$ is taken from T\'oth and Werner \cite{TW}. We notice by the way that here also the set 
$$\Lambda=\left\{(\Lambda_{a,v}(x),x\ge a)\right\}_{(a,v)\in \R\times [0,\infty)},$$
forms a family of reflected/absorbed coalescing processes. In \cite{TW} the $\Lambda_{a,v}$'s were moreover independent Brownian motions (reflected or absorbed in $0$ depending on the time interval) and therefore $\Lambda$ was called (in their Section 2.1) a FICRAB (for family of independent coalescing reflected and absorbed Brownian motions). Such family of coalescing Brownian motions seems to have been first studied by Arratia \cite{Arr} and is now better known under the name of Brownian web (see for instance \cite{FINR}).
Here the situation is slightly different: first each $\Lambda_{a,v}$ is some diffusion which is not a Brownian motion and before they coalesce two $\Lambda_{a,v}$'s are not independent. For instance if $v<v'$, then $(\Lambda_{a,v},\Lambda_{a,v'})$ satisfies the following system of stochastic differential equations:   
\begin{eqnarray}
\label{systemEDS}
\left\{
\begin{array}{lcl}
d\Lambda_{a,v}(x) &= & 2\sqrt{\Lambda_{a,v}(x)}\, dB_x + 2(1_{\{a\le x\le 0\}}+h(\Lambda_{a,v}(x)))\, dx \\ 
d\Lambda_{a,v'}(x) & = & 2\sqrt{\Lambda_{a,v}(x)}\, dB_x +2\sqrt{\Lambda_{a,v'}(x)-\Lambda_{a,v}(x)}\, d\widetilde B_x \\ 
 & & + 2(1_{\{a\le x \le 0\}}+h(\Lambda_{a,v'}(x)))\, dx,
\end{array}
\right.
\end{eqnarray}
for all $x\in [a,+\infty)$, where $B$ and $\widetilde B$ are two independent Brownian motions. This result follows from \eqref{la'v'lav} and the Ray-Knight theorem (see for instance \cite[Theorem 6.1]{RS}). 
Note that we could describe similarly the law of $(\Lambda_{a_i,v_i},i\in I)$, for any finite set $I$, and any $(a_i, v_i)$, $i\in I$. 
In \cite{TW}, the family $\Lambda$ was called a sequence of forward lines and the dual sequence, the sequence of backward lines 
$\Lambda^*=\left\{\Lambda^*_{a,v}(\cdot)\ :\ (a,v)\in \R\times [0,\infty)\right\}$, 
was defined by
\begin{eqnarray}
\label{lambda*}
\Lambda^*_{a,v}(x) =  \sup \, \{w\ :\ \Lambda_{-x,w}(-a) < v\},
\end{eqnarray}
for all $x\ge a$ and $v\ge 0$. 
As in \cite{TW} we can define $\Lambda^*$ here and we have also
$$(\Lambda^*_{a,v}(x),x\ge a)=(\Lambda_{-a,v}(-x),x\ge a),$$
in law (see Theorem 2.3 in \cite{TW}). It is important to observe that  
\begin{equation}
\label{lambda*lambda}
(\Lambda_{a,v}(x), x\in \R) \textrm{ is a function of }((\Lambda_{a,v}(x),x\ge a),(\Lambda^*_{-a,v}(x)), x\ge -a)).
\end{equation} 
We notice now some other notable differences with the situation in \cite{TW}. First if we denote by $Q_h$ the law of $\Lambda$, then the law of  $\Lambda^*$ is 
$Q_{-h}$. In particular $\Lambda$ and $\Lambda^*$ do not have the same law (in other words $\Lambda$ is not self-dual), except if $h=0$. 
Moreover, for any $a$ (say $a<0$) and $v\ge 0$, the process 
$\Lambda_{a,v}$ will almost surely not hit $0$ in the time interval $[a,0]$. The reason is that in the time interval $[0,\tau_a(v)]$ the excited BM will cross each level $x\in [a,0]$ and strictly increase its local time on these levels (by using the absolute continuity between the laws of a standard BM and the excited BM). 
Similarly given any $a<a'$, $v$ and $v'$, we have $\Lambda_{a,v}(x) \neq \Lambda^*_{-a',v'}(-x)=\Lambda_{a',v'}(x)$ for all $x\in [a,a']$ almost surely. 
Let us also notice that couples of processes such as $(\Lambda_{a,v}(x), a\le x\le 0)$ and  $(\Lambda_{0,v'}^*(x), 0\le x\le -a)$, if $a<0$, are \textit{conjugate diffusions} (see \cite{T3} for a definition). Similarly $(\Lambda_{0,v}(x), x\ge 0)$ and  $(\Lambda_{a,v'}^*(x), x\ge -a)$, if $a<0$, are also conjugate.  \\
Now we can sketch another proof of Theorem \ref{cvlawfd} which bypass the use of Corollary \ref{corext} and uses instead these notions of forward and backward lines. 
The idea is to first prove that
\begin{eqnarray}
\label{convlambda}
\left\{(\Lambda^{(n)}_{a_i,v_i}(x),x\ge a_i),i\in I\right\} \mathop{\Longrightarrow}_{n\rightarrow\infty}^{\mathcal{L}} \left\{(\Lambda_{a_i,v_i}(x),x\ge a_i),i\in I\right\}.
\end{eqnarray}  
This can be done by using Ethier-Kurtz's result (Theorem 4.1 p.354 in \cite{EK}), \eqref{egaliteloiaa} and \eqref{systemEDS}. One can next define analogues $\Lambda^{(n)}$ and $\Lambda^{(n),*}$ respectively of $\Lambda$ and $\Lambda^*$, in the discrete setting and it then suffices to use \eqref{lambda*} (and its discrete counterpart), \eqref{lambda*lambda} and \eqref{convlambda} to deduce the desired convergence. Since we already gave another proof, we omit the details here. 
}
\end{rem}

\end{document}